\documentclass[11pt]{amsart}

\usepackage{amsmath,amssymb,amsthm,mathrsfs,enumitem,booktabs,microtype}
\usepackage[margin=1in]{geometry}
\usepackage[hidelinks]{hyperref}

\numberwithin{equation}{section}
\setlist[enumerate]{itemsep=.35em,topsep=.5em}

\theoremstyle{plain}
\newtheorem{theorem}{Theorem}[section]
\newtheorem{proposition}[theorem]{Proposition}
\newtheorem{lemma}[theorem]{Lemma}
\newtheorem{corollary}[theorem]{Corollary}
\newtheorem{conjecture}[theorem]{Conjecture}
\theoremstyle{definition}

\newtheorem{example}[theorem]{Example}
\theoremstyle{remark}
\newtheorem{remark}[theorem]{Remark}

\newcommand{\bR}{\mathbb R}
\newcommand{\QQ}{\mathbb Q}
\newcommand{\D}{\mathcal D}
\newcommand{\Res}{\operatorname{Res}}
\newcommand{\disc}{\operatorname{disc}}

\newcommand{\Rey}{\operatorname{Rey}}
\newcommand{\tr}{\operatorname{tr}}

\newcommand{\Inj}{\operatorname{Inj}}
\newcommand{\barx}{\overline{x}}
\newcommand{\Sn}{\mathfrak{S}_n}

\title[Discriminants of derivatives]
{Discriminants of derivatives:\\ square graphs and sums of squares}

\author{Cordian Riener}
\address{Department of Mathematics and Statistics, UiT -- The Arctic University of Norway, N 9037 Troms\o, Norway}
\email{cordian.riener@uit.no}

\author{Boris Shapiro}
\address{Department of Mathematics, Stockholm University,
S--10691 Stockholm, Sweden}
\email{shapiro@math.su.se}

\begin{document}

\begin{abstract}
The discriminants of the derivatives of a real-rooted polynomial are
nonnegative polynomials in its roots, and we ask whether this
nonnegativity can be certified algebraically in the root differences.
We compare three such certificates: sums of squares, positive
combinations of symmetrized monomials in squared root differences (the
square-graph cone of Alexandersson and Shapiro), and sum-of-squares
factorizations of the Hermite matrix of the derivative. We work in the
terminal range, where the degree $r=n-k$ of the $k$-th derivative is
fixed and the degree $n$ of the polynomial grows. For $r=3$ the
discriminant lies in the square-graph cone for every $n$, and we prove
the three-graph formula conjectured by Alexandersson and Shapiro. For
$r=4$ the two cones separate: the discriminant belongs to the
square-graph cone exactly for $4\le n\le22$, but it is a sum of squares
for every $n\ge4$. The negative part is certified by an explicit linear
functional, the positive part by a rational Gram certificate built from
nine orbits of squares under the symmetric group. Thus the square-graph
conjecture fails, already at $(n,k)=(23,19)$, while the sum-of-squares
question has a positive answer for all $k\ge n-4$. Finally, we give
explicit sum-of-squares factorizations of the Hermite matrices of the
terminal cubic and quartic. The quartic factorization yields a uniform
proof of the sum-of-squares property for all $n\ge4$, and we conjecture
that such a factorization exists for every derivative; this is verified
in the centered outlier limit through terminal degree six.

\end{abstract}

\subjclass[2020]{Primary 26C10, 05C31; Secondary 12D10, 05E05, 14M25, 90C22}

\keywords{discriminant, derivative of a polynomial, symmetrized graph monomial,
sum of squares, square-graph cone, central moments, Hermite matrix,
hyperbolic polynomial}

\maketitle

\section{Introduction}

We consider the monic univariate polynomial  \[P(t)=\prod_{i=1}^n(t-x_i)\] of degree $n$,
whose roots $x_1,\ldots,x_n$ we regard as independent variables. If the
$x_i$ are real, then by Rolle's theorem every derivative $P^{(k)}$ is
real-rooted as well, and therefore its discriminant
\[
\D_k^{(n)}:=\disc\bigl(P^{(k)}\bigr),\qquad 0\le k\le n-2,
\]
is nonnegative. Throughout, discriminants carry the usual
leading-coefficient normalization: if $f(t)=c\prod_{j=1}^m(t-\rho_j)$,
then $\disc(f)=c^{2m-2}\prod_{i<j}(\rho_i-\rho_j)^2$, or equivalently
$\disc(f)=(-1)^{m(m-1)/2}c^{-1}\Res(f,f')$; see \cite{GKZ}.

As a function of the roots, $\D_k^{(n)}$ is a symmetric polynomial. It is
moreover unchanged when all roots are translated by the same amount, and
hence it is a polynomial in the differences $x_i-x_j$. For $k=0$ its
nonnegativity has a transparent explanation: the classical formula
$\D_0^{(n)}=\prod_{i<j}(x_i-x_j)^2$ exhibits the discriminant as a square
of a polynomial in the differences. For $k\ge1$ no such formula is
available, and it is natural to ask whether the nonnegativity of
$\D_k^{(n)}$ nevertheless admits an algebraic certificate in the root
differences. This question was raised by Sottile and Mukhin for the first
derivative, motivated by reality questions in Schubert calculus
\cite{So1}, and was answered by Sanyal, Sturmfels and Vinzant, who showed
via the entropic discriminant that $\disc(P')$ is a sum of squares of
polynomials in the differences of the roots \cite[Corollary~14]{Sanyal}.
In this paper we study the following two versions of the question for
all derivatives; the second is a strengthening of the first.

\begin{conjecture}[sum-of-squares question]\label{conj:sos}
For every $n\ge2$ and every $0\le k\le n-2$, the polynomial $\D_k^{(n)}$
is a sum of squares of polynomials in the differences $x_i-x_j$.
\end{conjecture}

\begin{conjecture}[square-graph cone conjecture, \cite{AS}]
\label{conj:squaregraph}
For every $n\ge2$ and every $0\le k\le n-2$, the polynomial $\D_k^{(n)}$
is a nonnegative linear combination of symmetrized products
\[
\prod_{(i,j)\in E(H)}(x_i-x_j)^2,
\]
where $H$ ranges over loopless multigraphs on the vertex set
$\{1,\ldots,n\}$ and edges are counted with multiplicity.
\end{conjecture}

Each symmetrized product in Conjecture~\ref{conj:squaregraph} is itself a
sum of squares of monomials in the differences, so the second conjecture
implies the first. Alexandersson and Shapiro found considerable evidence
for the stronger statement: \cite[Example~3]{AS} proposes an explicit
three-graph formula for $\disc(P^{(n-3)})$, Examples~4--5 there give
certificates for $\disc(P^{(n-4)})$ with $n=5,6$, and further certificates
are collected in \cite{AScert}.

Our main result is that the two conjectures part ways at a sharp
threshold. We work in the terminal range, in which the degree $r=n-k$ of
the derivative is fixed while $n$ grows. There it is convenient to
replace $P^{(n-r)}$ by the monic polynomial
\[
J_{n,r}(t)=\binom nr^{-1}\sum_{|S|=r}\prod_{i\in S}(t-x_i),
\]
the average of the characteristic polynomials of the $r$-subsets of the
roots, which differs from $P^{(n-r)}$ by a positive constant factor
(Proposition~\ref{prop:terminal-normalized}). Accordingly, $\D_{n-r}^{(n)}$
is a positive multiple of $\Delta_{n,r}:=\disc(J_{n,r})$, a polynomial of
degree $r(r-1)$ in the roots whose coefficients, in centered power-sum
coordinates, are rational functions of $n$. Membership in either cone is
unaffected by positive scalars, so we work with $\Delta_{n,r}$ throughout.

We write $\mathcal C_{n,d}$ for the cone generated by the symmetrized
square-graph monomials of degree $d$ in $n$ variables, and $\Sigma_{n,d}$
for the cone of symmetric forms of degree $d$ that are sums of squares of
polynomials in the differences $x_i-x_j$. Our first result concerns the
terminal cubic, where the square-graph conjecture holds in a strong form.

\begin{theorem}\label{thm:main-cubic}
For every $n\ge3$ the discriminant $\Delta_{n,3}$ belongs to
$\mathcal C_{n,6}$. For $n\ge5$ three graph types suffice, and the
resulting formula is the one conjectured in \cite[Example~3]{AS}.
\end{theorem}

Explicitly, for $n\ge5$,
\begin{equation}\label{eq:intro-cubic-threegraph}
\Delta_{n,3}=\frac{108}{n^3(n-1)^3}
\left(\frac{n+3}{n(n-2)}T_n
+\frac{U_n}{n(n-2)(n-4)}
+\frac{R_n}{8n(n-2)^2}\right),
\end{equation}
where
\[
\begin{aligned}
T_n&=\sum_{1\le i<j<k\le n}
 (x_i-x_j)^2(x_j-x_k)^2(x_k-x_i)^2,\\
U_n&=\sum_{i,j,k,l,m}^{*}
 (x_i-x_j)^2(x_j-x_k)^2(x_l-x_m)^2,\\
R_n&=\sum_{i,j,k,l,m,s}^{*}
 (x_i-x_j)^2(x_k-x_l)^2(x_m-x_s)^2,
\end{aligned}
\]
and the asterisk denotes summation over ordered pairwise distinct
indices. The three sums are the symmetrizations of a doubled triangle, of
a doubled two-edge path together with a disjoint doubled edge, and of a
matching of three doubled edges; at $n=5$ the last sum is empty.
Section~\ref{sec:cubic} also contains a four-graph formula valid for all
$n\ge3$ and the conversion to the symmetrization convention of \cite{AS}.

The terminal quartic behaves differently, and this is the central
observation of the paper.

\begin{theorem}\label{thm:main-quartic}
Let $n\ge4$. Then $\Delta_{n,4}\in\mathcal C_{n,12}$ if and only if
$4\le n\le22$. For $n\ge23$ there is an explicit linear functional
$\Lambda_n$ which is nonnegative on $\mathcal C_{n,12}$ and negative on
$\Delta_{n,4}$. Nevertheless, $\Delta_{n,4}\in\Sigma_{n,12}$ for every
$n\ge4$, and for $n\ge23$ this is witnessed by a certificate consisting
of nine orbits of squares under $\Sn$ with coefficients in $\QQ(n)$.
\end{theorem}

In terms of the two conjectures, Conjecture~\ref{conj:squaregraph} holds
for $k=n-2$ and $k=n-3$ but fails for $k=n-4$ as soon as $n\ge23$, the
smallest failing pair in this family being $(n,k)=(23,19)$, whereas
Conjecture~\ref{conj:sos} holds for all $k\ge n-4$. Since $\Lambda_n$ is
negative on the sum of squares $\Delta_{n,4}$, it also shows that
$\mathcal C_{n,12}$ is a proper subcone of $\Sigma_{n,12}$ for $n\ge23$.
That the two cones differ is not surprising in itself, and can already be
seen in degree four (Remark~\ref{rem:degree-four-separation}). What is
remarkable about the family $\Delta_{n,4}$ is that a distinguished
polynomial, nonnegative for every $n$, crosses from one cone into the
other at a definite value of $n$.

The two theorems are proved by different means.
Theorem~\ref{thm:main-cubic} rests on explicit identities between
symmetrized graph monomials and central moments of the roots. In
Theorem~\ref{thm:main-quartic}, both the certificates for $n\le22$ and
the separating functional for $n\ge23$ come from a linear problem of
fixed size: in fixed degree only finitely many graph types occur, and
their coordinates are polynomials in $n$, so membership in
$\mathcal C_{n,12}$ is a linear feasibility problem with data polynomial
in $n$ (Lemma~\ref{lem:finite-types}). For the sum-of-squares statement
we use the symmetry of $\Delta_{n,4}$ to break the Gram matrix into
blocks indexed by irreducible representations of $\Sn$
\cite{GatermannParrilo}; in fixed degree the sizes of these blocks do
not depend on $n$ \cite{Riener13,BlekhermanRiener}. The vanishing of the
discriminant at configurations with many coinciding roots cuts the
blocks down further, and the certificate we found uses only four of
them, of sizes $1,1,3,4$. A block of size $m$ contributes $m$ squares,
each averaged over $\Sn$, which accounts for the nine orbits in the
theorem. The certificate is a single identity over $\QQ(n)$, and its
positivity for all real $n\ge23$ is checked at once by exact rational
arithmetic. Section~\ref{sec:quartic-sos} and
Appendix~\ref{app:verification} give the matrix definitions and exact
certificate data. Numerical calculations were used to find the
certificate; its proof consists of identities and positivity checks
over $\QQ(n)$.

The certificates above establish the sum-of-squares property without
explaining it. Our third result offers an explanation for the terminal
cubic and quartic, through the Hermite matrix. For a monic polynomial
$f$ of degree $r$, the Hermite matrix $H(f)$ is the Hankel matrix of the
Newton sums of its roots; its entries are polynomials in the
coefficients of $f$, and $\det H(f)=\disc(f)$. If $H(f)=A^{\mathsf T}A$
for a polynomial matrix $A$, then the Cauchy--Binet formula expresses
$\disc(f)$ as the sum of the squares of the maximal minors of $A$.
Section~\ref{sec:hermite} develops this mechanism, which for the first
derivative recovers the symmetric-matrix construction of
\cite{Ilyushechkin,Parlett,Domokos}, and proves the following.

\begin{theorem}\label{thm:main-hermite}
For $r\in\{3,4\}$ and every $n\ge r$, the Hermite matrix of $J_{n,r}$
is a sum-of-squares matrix over the ring of polynomials in the centered
roots $z_i=x_i-\frac1n\sum_jx_j$. Consequently,
$\Delta_{n,r}\in\Sigma_{n,r(r-1)}$.
\end{theorem}

The factorizations are explicit (Theorems~\ref{thm:cubic-hermite}
and~\ref{thm:quartic-hermite}), and the cubic one yields a short scalar
identity (Corollary~\ref{cor:cubic-hermite-scalar}). The quartic one
deserves emphasis: it proves that $\Delta_{n,4}$ is a sum of squares for
all $n\ge4$ by a single identity, independently of the graph
certificates and of the nine-orbit Gram certificate, and it settles the
first instance of the Hermite-matrix question beyond the original
polynomial, its first derivative and the terminal quadratic and cubic,
namely $(n,k)=(6,2)$. We conjecture that the Hermite matrix of every
monic derivative $P^{(k)}/\operatorname{lc}(P^{(k)})$ is a
sum-of-squares matrix (Conjecture~\ref{conj:hermite}); the first open
case is $(n,k)=(7,2)$.

The table below summarizes the state of the three questions in the
terminal families; throughout, $n\ge r$ and $k=n-r$.
\begin{center}
\small
\begin{tabular}{@{}c p{3.2cm} p{3.2cm} p{5cm}@{}}
\toprule
$r$ & Square graphs & Scalar SOS & Hermite-matrix SOS\\
\midrule
$2$ & All $n\ge2$ & All $n\ge2$ & All $n\ge2$\\
$3$ & All $n\ge3$ & All $n\ge3$ & All $n\ge3$\\
$4$ & Exactly $4\le n\le22$ & All $n\ge4$ & All $n\ge4$\\
$r\ge5$ & Open in general & Open in general & Open in general\\
\bottomrule
\end{tabular}
\end{center}
For every $r$ the cases $n=r$ and $n=r+1$ are the original polynomial
and its first derivative; both have Hermite-matrix factorizations, and
the former also has its classical square-graph formula. In the centered
outlier limit $n\to\infty$, the Hermite matrices are sum-of-squares
matrices for $3\le r\le6$ (Proposition~\ref{prop:limit-hermite}). For
$r=4$ the finite-$n$ factorization of Theorem~\ref{thm:main-hermite} is
available alongside this limit, whereas for $r=5,6$ a limiting
certificate alone does not decide the finite-$n$ question.

The paper is organized as follows. Section~\ref{sec:terminal}
introduces the normalized terminal polynomials $J_{n,r}$, expresses
them through central moments of the roots, derives the finite Appell
relation between them, and describes the real zero locus of their
discriminants. Section~\ref{sec:cubic} proves the cubic results,
Section~\ref{sec:threshold} the square-graph threshold for the quartic,
and Section~\ref{sec:quartic-sos} the nine-orbit certificate.
Section~\ref{sec:hermite} develops the Hermite-matrix mechanism, proves
the cubic and quartic factorizations, and states the general conjecture.
Section~\ref{sec:limits} shows that in any fixed terminal degree each
cone-membership question has the same answer for all sufficiently large
$n$, verifies the matrix conjecture in the centered outlier limit
through terminal degree six, and collects open questions.
Appendix~\ref{app:quintic} records the terminal quintic,
Appendix~\ref{app:outliers} the explicit outlier-limit factorizations,
and Appendix~\ref{app:verification} the coefficient matrices and
supplementary certificate data.

\section{Graph monomials and terminal polynomials}\label{sec:terminal}

In this section we collect the basic objects and formulas that will be
used throughout the paper. We first recall the language of graph
monomials, which provides the natural ambient vector space for the
discriminants $\D_k^{(n)}$. We then normalize the terminal derivatives,
express them through subsets of the roots and through central moments,
work out the cases $r\le4$ explicitly, and describe the real zero locus
of their discriminants by means of interlacing.

\subsection{Graph monomials and the two cones}
Symmetric translation-invariant polynomials in the roots admit a
convenient combinatorial description, which we now recall. Let $G$ be a
loopless multigraph on the vertex set $\{1,\ldots,n\}$. To $G$ we
associate its \emph{graph monomial}
\[
M_G(x_1,\ldots,x_n)=\prod_{(i,j)\in E(G)}(x_i-x_j),
\]
where each edge is given an orientation; a different choice of
orientations changes $M_G$ at most by a sign. The \emph{symmetrized
graph monomial} of $G$ is
\[
\widetilde G:=\sum_{\sigma\in \Sn}M_G(x_{\sigma(1)},\ldots,x_{\sigma(n)}),
\]
where here and throughout $\Sn$ denotes the symmetric group, acting on
$\{1,\ldots,n\}$ and on polynomials by permuting the variables. We
observe that if every edge of $G$ occurs with even multiplicity, then
each summand of $\widetilde G$ is the square of a monomial in the
differences of the variables, and hence $\widetilde G$ is a sum of
squares. We will call such graphs \emph{square graphs}; the cone they
generate in a given degree is the square-graph cone $\mathcal C_{n,d}$
of the introduction.

Let $\mathrm{PST}_{n,d}$ denote the vector space of homogeneous
symmetric polynomials of degree $d$ in $n$ variables which are
translation invariant, i.e., satisfy $f(x+t\mathbf1)=f(x)$ for all
$t\in\bR$. It is classical that this space is spanned by the symmetrized
graph monomials of graphs with $d$ edges; we refer the reader to
\cite{AS} and the references given there. The discriminants we study
live in these spaces. Indeed, we have already seen that $\D_k^{(n)}$ is
symmetric and translation invariant in the roots, and it is also
homogeneous: scaling all roots by $a\ne0$ gives $P(t;ax)=a^nP(t/a;x)$,
so every root of $P^{(k)}$ is scaled by $a$ as well, while the leading
coefficient of $P^{(k)}$ does not depend on the roots at all. Hence
$\D_k^{(n)}$ is homogeneous of degree
\[
d:=\deg \D_k^{(n)}=(n-k)(n-k-1),
\]
the number of ordered pairs of distinct roots of $P^{(k)}$. In this
language, Conjecture~\ref{conj:squaregraph} asks whether
$\D_k^{(n)}\in\mathrm{PST}_{n,d}$ lies in the square-graph cone
$\mathcal C_{n,d}$, while Conjecture~\ref{conj:sos} asks whether it lies
in the larger cone $\Sigma_{n,d}$ of symmetric forms that are sums of
squares of polynomials in the differences.

At first sight, the requirement in Conjecture~\ref{conj:sos} that the
squares be polynomials in the differences looks like a restriction. The
following simple observation shows that it is automatic.

\begin{lemma}\label{lem:translation}
Let $f\in\bR[x_1,\ldots,x_n]$ be translation invariant, i.e.,
$f(x+t\mathbf1)=f(x)$ for all $t\in\bR$. If $f=\sum_jg_j^2$ with
$g_j\in\bR[x]$, then every $g_j$ is translation invariant, and hence a
polynomial in the differences $x_i-x_j$.
\end{lemma}

\begin{proof}
Fix $x\in\bR^n$ and consider the univariate polynomial
$t\mapsto\sum_jg_j(x+t\mathbf1)^2$. By translation invariance of $f$ it
is the constant $f(x)$, so each $g_j(x+t\mathbf1)$ is a bounded
polynomial in $t$ and therefore constant. Hence every $g_j$ is
translation invariant, and substituting $t=-x_n$ shows
$g_j(x)=g_j(x_1-x_n,\ldots,x_{n-1}-x_n,0)$, which is a polynomial in
the differences.
\end{proof}

\subsection{The terminal chain and its discriminants}
It will be convenient to work with centered roots. We put
\[
\barx=\frac1n\sum_{i=1}^n x_i,\qquad
z_i=x_i-\barx,\qquad
p_j=\sum_{i=1}^n z_i^j\quad(j\ge1),
\]
so that $p_1=0$, and we denote by $e_j$ the elementary symmetric
polynomials in the roots. In the terminal range it is natural to
normalize the derivative $P^{(n-r)}$ to be monic. The resulting
polynomial has a pleasant combinatorial description, which we record
together with its immediate consequences.

\begin{proposition}[subset-average and Appell structure]
\label{prop:terminal-normalized}
Fix $r\ge2$ and $n\ge r$, and put
\[
J_{n,r}(t)=\frac{r!}{n!}P^{(n-r)}(t).
\]
\begin{enumerate}[label={\rm(\alph*)}]
\item $J_{n,r}$ is the average of the characteristic polynomials of the
$r$-subsets of the roots:
\begin{equation}\label{eq:subset-average}
J_{n,r}(t)=\binom nr^{-1}
\sum_{\substack{S\subseteq\{1,\ldots,n\}\\ |S|=r}}
\prod_{i\in S}(t-x_i).
\end{equation}
\item If the roots are centered, then
\begin{equation}\label{eq:terminal-coefficients}
J_{n,r}(t)
=t^r+\sum_{j=2}^r(-1)^j
\frac{\binom rj}{\binom nj}e_jt^{r-j}.
\end{equation}
In particular, $J_{n,r}$ is a polynomial in $t$ and $p_2,\ldots,p_r$
with coefficients in $\QQ(n)$, and
\[
\D_{n-r}^{(n)}=\left(\frac{n!}{r!}\right)^{2r-2}\disc(J_{n,r}).
\]
\item With $J_{n,1}(t)=t-\barx$ and $J_{n,0}(t)=1$, the polynomials
$J_{n,r}$ satisfy
\begin{equation}\label{eq:appell}
J'_{n,r}(t)=rJ_{n,r-1}(t),\qquad 1\le r\le n,
\end{equation}
and consequently
\begin{equation}\label{eq:terminal-resultant}
\disc(J_{n,r})
=(-1)^{r(r-1)/2}r^r\Res(J_{n,r},J_{n,r-1}).
\end{equation}
\item If the roots $x_i$ are real, then every $J_{n,r}$ is real-rooted,
and the roots of $J_{n,r}$ and $J_{n,r-1}$ interlace.
\end{enumerate}
\end{proposition}

\begin{proof}
(a) When the product $P(t)=\prod_{i=1}^n(t-x_i)$ is differentiated
$n-r$ times, each product $\prod_{i\in S}(t-x_i)$ over an $r$-subset
$S$ arises exactly $(n-r)!$ times, so that
\[
P^{(n-r)}(t)=(n-r)!\sum_{|S|=r}\prod_{i\in S}(t-x_i).
\]
Multiplying by $r!/n!=\bigl(\binom nr(n-r)!\bigr)^{-1}$ gives
\eqref{eq:subset-average}.

(b) In the average \eqref{eq:subset-average}, the term
$(-1)^je_jt^{r-j}$ appears with coefficient equal to the probability
that a fixed $j$-subset of the roots is contained in a uniformly chosen
$r$-subset, namely $\binom{n-j}{r-j}/\binom nr=\binom rj/\binom nj$.
This gives \eqref{eq:terminal-coefficients}. Since $e_1=p_1=0$, Newton's
identities express each $e_j$ as a polynomial in $p_2,\ldots,p_j$ with
rational coefficients, and the formula for $\D_{n-r}^{(n)}$ follows from
$\disc(cf)=c^{2r-2}\disc(f)$ for a polynomial $f$ of degree $r$.

(c) Differentiating the definition of $J_{n,r}$ gives
$J'_{n,r}=\frac{r!}{n!}P^{(n-r+1)}=rJ_{n,r-1}$. Since $J_{n,r}$ is monic,
$\disc(J_{n,r})=(-1)^{r(r-1)/2}\Res(J_{n,r},J'_{n,r})$, and
\eqref{eq:terminal-resultant} follows because the resultant is
homogeneous of degree $r$ in the coefficients of its second argument.

(d) By Rolle's theorem, the roots of a real-rooted polynomial and those
of its derivative interlace. Applying this to the successive derivatives
of $P$ gives the claim for the whole chain.
\end{proof}

Relation \eqref{eq:appell} says that, for each fixed $n$, the sequence
$J_{n,0},J_{n,1},\ldots,J_{n,n}=P$ is an Appell sequence in the
classical sense \cite{Appell}, cut off at degree $n$. We will refer to
it as a \emph{finite Appell chain}.

We write $J_{n,r}(t;x)$ when the roots need to be displayed. From now
on we assume, unless stated otherwise, that the roots are centered:
$J_{n,r}(t)$ stands for $J_{n,r}(t;z)=J_{n,r}(t+\barx;x)$, and in
particular $J_{n,1}(t)=t$. For the discriminant
\[
\Delta_{n,r}:=\disc(J_{n,r})
\]
nothing is lost by this convention, since it is translation invariant.
The first members of the chain are
\[
J_{n,2}(t)=t^2-\frac{p_2}{n(n-1)},
\qquad
J_{n,3}(t)=t^3-\frac{3p_2}{n(n-1)}t
-\frac{2p_3}{n(n-1)(n-2)},
\]
and $J_{n,4}=q_{n,4}$ is given in Proposition~\ref{prop:quartic} below.
Note that $q'_{n,4}=4J_{n,3}$ by \eqref{eq:appell}, so the cubic moment
inequality of Proposition~\ref{prop:cubic} also governs the critical
points of the terminal quartic. The terminal quintic and an odd slice
of it are recorded in Appendix~\ref{app:quintic}.

\begin{remark}\label{rem:n-dependence}
Let $m_j=e_j/\binom nj$ denote the average of the products
$x_{i_1}\cdots x_{i_j}$ over the $j$-subsets of the roots. In these
coordinates \eqref{eq:terminal-coefficients} no longer involves $n$:
for centered roots,
\[
J_{n,r}(t)=\sum_{j=0}^r(-1)^j\binom rjm_jt^{r-j},
\qquad\text{e.g.}\quad
J_{n,4}(t)=t^4+6m_2t^2-4m_3t+m_4 .
\]
Hence $\Delta_{n,r}$ is a single polynomial
$\Delta_r(m_2,\ldots,m_r)$, the same for all $n$, evaluated on the
means of the roots; only the map $z\mapsto(m_2,\ldots,m_r)$ depends on
$n$. The polynomial $\Delta_r$ is not nonnegative on $\bR^{r-1}$, since
not every polynomial of degree $r$ is real-rooted: for instance
$t^4-t$ has discriminant $-27$. So no certificate can be written in the
$m_j$ alone. Nonnegativity, and every certificate in this paper, is a
property of the composite map, which is where $n$ enters.
\end{remark}

We now work out the cases $r\le4$ explicitly. The smallest terminal
degree, $r=2$, can be understood by hand, and it shows the simplest
instance of the pattern we are after: the discriminant is a positive
multiple of a single square-graph sum.

\begin{example}[Quadratic terminal case]\label{ex:quadratic}
For every $n\ge2$,
\[
\D_{n-2}^{(n)}
=(n-1)((n-2)!)^2\sum_{1\le i<j\le n}(x_i-x_j)^2.
\]
Equivalently,
\[
\D_{n-2}^{(n)}=n(n-1)((n-2)!)^2p_2.
\]
\end{example}

\begin{proof}
The polynomial $P^{(n-2)}$ is quadratic.  Its coefficients are
\[
P^{(n-2)}(t)=\frac{n!}{2}t^2+(n-1)!a_1t+(n-2)!a_2.
\]
Therefore
\[
\disc(P^{(n-2)})=((n-1)!a_1)^2-2n!(n-2)!a_2.
\]
Using $a_1=-\sum_i x_i$ and $a_2=\sum_{i<j}x_ix_j$, this becomes
\[
(n-1)((n-2)!)^2\left((n-1)\Big(\sum_i x_i\Big)^2
-2n\sum_{i<j}x_ix_j\right).
\]
The expression in parentheses is
$\sum_{i<j}(x_i-x_j)^2=np_2$, which proves the claim.
\end{proof}

This proves Conjecture~\ref{conj:squaregraph} in the terminal quadratic case.
It uses the same square-graph ray as \cite[Example~2]{AS}; the coefficient
above follows from our stated discriminant normalization.

The cubic is the first nontrivial case; its discriminant has a compact
expression in the second and third central moments.

\begin{proposition}[the cubic terminal formula]\label{prop:cubic}
For every $n\ge3$,
\[
\D_{n-3}^{(n)}=\frac{n!((n-2)!)^3}{12}
\left(p_2^3-\frac{n(n-1)}{(n-2)^2}p_3^2\right).
\]
In particular, if the roots $x_1,\ldots,x_n$ are real, then
\[
 p_2^3\ge \frac{n(n-1)}{(n-2)^2}p_3^2 .
\]
\end{proposition}

\begin{proof}
Translate the variable so that the roots become $z_i=x_i-\barx$ and
$\sum_i z_i=0$.  Translation does not change the discriminant.
The polynomial $P$ then has the form
\[
P(t)=t^n+e_2t^{n-2}-e_3t^{n-3}+\cdots,
\]
where, by Newton's identities,
\[
e_2=-\frac{p_2}{2},\qquad e_3=\frac{p_3}{3}.
\]
After differentiating $n-3$ times, we obtain the depressed cubic
\[
P^{(n-3)}(t)=\frac{n!}{6}t^3+(n-2)!e_2t-(n-3)!e_3.
\]
For a depressed cubic $At^3+Bt+C$, the discriminant is
\[
-4AB^3-27A^2C^2.
\]
Substituting
\[
A=\frac{n!}{6},\qquad
B=-\frac{(n-2)!}{2}p_2,
\qquad
C=-\frac{(n-3)!}{3}p_3
\]
gives
\[
\D_{n-3}^{(n)}
=\frac{n!((n-2)!)^3}{12}p_2^3
-\frac{(n!)^2((n-3)!)^2}{12}p_3^2.
\]
Since
\[
\frac{(n!)^2((n-3)!)^2}{n!((n-2)!)^3}
=\frac{n(n-1)}{(n-2)^2},
\]
the stated formula follows.  If all $x_i$ are real, then $P^{(n-3)}$ is a real
cubic with three real roots by Rolle's theorem, and hence its discriminant is
nonnegative.  This gives the displayed moment inequality.
\end{proof}

\begin{remark}
The inequality in Proposition~\ref{prop:cubic} is the sharp finite version of the
usual bound on skewness for a distribution supported on $n$ equally weighted
points: with the sample skewness $g_1=\sqrt n\,p_3/p_2^{3/2}$ it reads
$g_1^2\le(n-2)^2/(n-1)$, the classical bound of Wilkins \cite{Wilkins}.  The
normalization here is especially natural because it is forced by the
discriminant of $P^{(n-3)}$.
\end{remark}

The quartic remains explicit, but now the fourth central moment enters.

\begin{proposition}[normalized quartic formula]\label{prop:quartic}
Let $n\ge4$ and
\[
q_{n,4}(t)=t^4+\alpha t^2+\beta t+\gamma,
\]
where
\[
\begin{aligned}
\alpha&=-\frac{6p_2}{n(n-1)},
&\beta&=-\frac{8p_3}{n(n-1)(n-2)},\\
\gamma&=\frac{3(p_2^2-2p_4)}{n(n-1)(n-2)(n-3)}.
\end{aligned}
\]
Then, after centering the roots,
\[
P^{(n-4)}(t)=\frac{n!}{24}\,q_{n,4}(t)
\]
and hence
\[
\D_{n-4}^{(n)}=\left(\frac{n!}{24}\right)^6\disc(q_{n,4}).
\]
In particular, for real roots the depressed quartic $q_{n,4}$ is hyperbolic
and
\begin{equation}\label{eq:disc4}
0\le \disc(q_{n,4})=256\gamma^3-128\alpha^2\gamma^2
+144\alpha\beta^2\gamma-27\beta^4+16\alpha^4\gamma-4\alpha^3\beta^2 .
\end{equation}
If $p_3=0$, this specializes to the factorization
\[
\disc(q_{n,4})=16\gamma(\alpha^2-4\gamma)^2 .
\]
\end{proposition}

\begin{proof}
Newton's identities in the centered variables give
\[
e_2=-\frac{p_2}{2},\qquad
e_3=\frac{p_3}{3},\qquad
e_4=\frac{p_2^2}{8}-\frac{p_4}{4}.
\]
Differentiating
$P(t)=t^n+e_2t^{n-2}-e_3t^{n-3}+e_4t^{n-4}+\cdots$
exactly $n-4$ times and dividing by $n!/24$ gives the displayed polynomial
$q_{n,4}$.  The discriminant formula is the standard one for a depressed
quartic, and the scaling identity follows from
$\disc(cq)=c^6\disc(q)$.  Rolle's theorem gives hyperbolicity for real
original roots, while $\beta=0$ gives the final biquadratic factorization.
\end{proof}

\begin{remark}
The quartic formula gives a natural continuation of the cubic skewness
inequality.  In the language of finite probability distributions with equal
weights it is an algebraic inequality among the variance, skewness, and fourth
central moment.  Unlike the cubic case, the equality locus is no longer only a
two-point configuration; it corresponds to configurations for which
$P^{(n-4)}$ has a multiple root, and these are determined next.
\end{remark}

The real zero locus of $\Delta_{n,r}$ is governed by interlacing.

\begin{lemma}[real zero locus and coincidence factorization]\label{lem:zero-locus}
Let $r\ge2$, $n\ge r$, and let $x_1,\ldots,x_n$ be real.  Then
$\Delta_{n,r}=0$ if and only if some value occurs among the $x_i$ with
multiplicity at least $n-r+2$.
Moreover, for $P(t)=(t-u)^{n-2}(t-v)(t-w)$ and $n\ge4$,
\[
\begin{aligned}
 \binom n4J_{n,4}(t;x)=(t-u)^2
 \Bigl[&\tbinom{n-2}4(t-u)^2+\tbinom{n-2}3(t-u)(2t-v-w)\\
 &+\tbinom{n-2}2(t-v)(t-w)\Bigr].
\end{aligned}
\]
\end{lemma}

\begin{proof}
Let $P$ be real-rooted, with distinct roots
$\rho_1<\cdots<\rho_m$ of multiplicities $\mu_1,\ldots,\mu_m$.  The
derivative $P'$ vanishes at $\rho_i$ to order $\mu_i-1$, and Rolle's theorem
gives at least one root in each interval $(\rho_i,\rho_{i+1})$.  These roots
already account for the full degree, since
\[
\sum_i(\mu_i-1)+(m-1)=\deg P-1=\deg P'.
\]
Thus each gap contains exactly one root of $P'$, and that root is simple.
Every multiple root of $P'$ therefore comes from a root of $P$ whose
multiplicity has dropped by one. All successive derivatives are real-rooted,
so this conclusion applies at every step. For $k\ge1$, suppose that $a$
is a root of $P^{(k)}$ of multiplicity $s\ge2$. Applying the preceding
observation backwards shows that $a$ is a root of $P^{(k-j)}$ of
multiplicity $s+j$ for every $0\le j\le k$. In particular, it is a root
of the original polynomial $P$ of multiplicity
$s+k\ge k+2$. Conversely, if $P$ has a root $a$ of multiplicity
$m\ge k+2$, differentiating $k$ times reduces its multiplicity exactly
to $m-k\ge2$. The case $k=0$ is immediate. Thus $P^{(k)}$ has a multiple
root if and only if $P$ has a root of multiplicity at least $k+2$.
Taking $k=n-r$ proves the first assertion.

For the displayed factorization, group the terms in the subset average
\eqref{eq:subset-average} according to whether the chosen four-subset contains
four, three, or two copies of $u$.  Their respective multiplicities are
$\binom{n-2}4$, $2\binom{n-2}3$, and $\binom{n-2}2$.
\end{proof}

\begin{corollary}[cubic equality case]\label{cor:cubic-equality}
Assume $x_1,\ldots,x_n\in\bR$ are not all equal.  Equality in
\[
 p_2^3\ge \frac{n(n-1)}{(n-2)^2}p_3^2
\]
holds if and only if, after a permutation, the centered roots have the form
\[
z_1=\cdots=z_{n-1}=a,
\qquad z_n=-(n-1)a
\]
for some nonzero $a\in\bR$.
\end{corollary}

\begin{proof}
By Proposition~\ref{prop:cubic}, equality is equivalent to
$\Delta_{n,3}=0$.  Lemma~\ref{lem:zero-locus} says that one root value must
then have multiplicity at least $n-1$, and centering gives the displayed
configuration.  The converse is immediate.
\end{proof}

For $r=2$, Lemma~\ref{lem:zero-locus} says that $p_2=0$ only for constant
configurations.  For $r=4$, the zero locus consists of configurations with a
root of the original polynomial $P$ of multiplicity at least $n-2$:
after $n-4$ derivatives, its multiplicity is still at least two.

The zero locus gives linear restrictions on every polynomial appearing
in an SOS. These value restrictions need not determine the smallest
feasible Gram face. At more singular coincidence configurations, the
order of vanishing imposes additional restrictions; see
Lemma~\ref{lem:quartic-vanishing-order}.

\section{A square-graph certificate for the terminal cubic}\label{sec:cubic}

The quadratic and cubic terminal families remain uniformly inside the
square-graph cone.  The cubic case is subtler than the quadratic one: the
formula of Proposition~\ref{prop:cubic} is not itself a square-graph
decomposition, because it contains the difference of two invariant terms.
The essential point is therefore to prove that
\[
        p_2^3-\frac{n(n-1)}{(n-2)^2}p_3^2
\]
belongs to the square-graph cone.  This section gives an explicit positive
certificate, which proves Conjecture~\ref{conj:squaregraph} for the terminal
cubic family $k=n-3$.  The certificate is stronger than the corresponding
moment inequality, because every summand is a symmetrized monomial in
squared root differences.

Continue to use the centered variables
\[
 z_i=x_i-\bar x,
 \qquad \sum_i z_i=0,
 \qquad p_r=\sum_i z_i^r .
\]
Recall the three square-graph sums $T_n,U_n,R_n$ defined in the
introduction. They are unchanged when $x_i$ is replaced by $z_i$.
We also use the following ordered four-vertex sums:
\[
\begin{aligned}
P_n&=\sum_{i,j,k,l}^{*}
 (z_i-z_j)^2(z_j-z_k)^2(z_k-z_l)^2,\\
Q_n&=\sum_{i,j,k,l}^{*}
 (z_i-z_j)^4(z_k-z_l)^2.
\end{aligned}
\]
As before, $\sum^*$ means summation over ordered pairwise distinct indices.
Here $P_n$ is the ordered path sum and $Q_n$ is the ordered sum of a
double edge disjoint from a single edge. Each of these sums belongs to
the square-graph cone.

\begin{lemma}\label{lem:degree-six-sums}
For centered variables $z_1,\ldots,z_n$ one has
\[
\begin{aligned}
T_n={}&-p_2^3-np_3^2+np_2p_4,\\
P_n={}&(4n-6)p_2^3-2(n-5)(n-2)p_3^2\\
&{} +(4n^2-24n+30)p_2p_4-2n(n-1)p_6,\\
Q_n={}&12(n-2)p_2^3+8(2n-5)p_3^2\\
&{} +(4n^2-36n+60)p_2p_4-4n(n-1)p_6,\\
R_n={}&8(n-2)(n^2-7n+15)p_2^3
 -16(3n^2-15n+20)p_3^2\\
&{} -24(n-2)(n^2-5n+10)p_2p_4
 +16n(n-1)(n-2)p_6 .
\end{aligned}
\]
\end{lemma}

\begin{proof}
We record the diagonal-removal calculation.  It is a useful check on the
normalization of the ordered sums.  For nonnegative integers
$\alpha_1,\ldots,\alpha_m$, put
\[
S(\alpha_1,\ldots,\alpha_m)=
\sum_{i_1,\ldots,i_m}^{*} z_{i_1}^{\alpha_1}\cdots z_{i_m}^{\alpha_m}.
\]
By inclusion--exclusion, or equivalently by M\"obius inversion on the lattice
$\Pi_m$ of set partitions, one has
\begin{equation}\label{eq:mobius-distinct}
\begin{aligned}
S(\alpha_1,\ldots,\alpha_m)
&=\sum_{\pi\in\Pi_m}\mu(\pi)
  \prod_{B\in\pi} p_{\sum_{j\in B}\alpha_j},\\
\mu(\pi)&=\prod_{B\in\pi}(-1)^{|B|-1}(|B|-1)! .
\end{aligned}
\end{equation}
Here $p_0=n$ and $p_1=0$.  Formula~\eqref{eq:mobius-distinct} is just the
standard principle that the unrestricted sum is decomposed according to the
partition of positions on which the indices coincide.

Expand each of the four square-graph products as a sum of monomials in the
corresponding ordered variables and apply~\eqref{eq:mobius-distinct}.  For the
triangle one first computes the ordered sum and divides by $6$, because each
unordered triple is counted six times; this gives
\[
T_n=-p_2^3-np_3^2+np_2p_4.
\]
The same expansion for the ordered path, the ordered disjoint double-edge plus
single-edge graph, and the ordered matching gives the displayed expressions for
$P_n,Q_n$ and $R_n$.  No specialization of the variables is used in this
calculation.
\end{proof}

\begin{proposition}[four-graph cubic expansion]\label{prop:cubic-square-expansion}
For every $n\ge3$,
\[
\begin{aligned}
 p_2^3-\frac{n(n-1)}{(n-2)^2}p_3^2
={}&\frac{6}{n(n-2)}T_n
 +\frac{1}{2n(n-2)}P_n\\
&{} +\frac{1}{4n(n-2)}Q_n
 +\frac{1}{8n(n-2)^2}R_n .
\end{aligned}
\]
Consequently $\D_{n-3}^{(n)}$ belongs to the square-graph cone for every
$n\ge3$.
\end{proposition}

\begin{proof}
Substitute the four identities of Lemma~\ref{lem:degree-six-sums} into the
right-hand side.  The coefficients of $p_2p_4$ and $p_6$ cancel.  The
coefficient of $p_2^3$ is
\[
-\frac{6}{n(n-2)}+\frac{4n-6}{2n(n-2)}
+\frac{12(n-2)}{4n(n-2)}
+\frac{8(n-2)(n^2-7n+15)}{8n(n-2)^2}=1,
\]
and the coefficient of $p_3^2$ is
\[
-\frac{6}{n-2}-\frac{n-5}{n}
+\frac{2(2n-5)}{n(n-2)}
-\frac{2(3n^2-15n+20)}{n(n-2)^2}
=-\frac{n(n-1)}{(n-2)^2}.
\]
This proves the displayed identity.  All four coefficients in the identity are
positive for $n\ge3$, and $T_n,P_n,Q_n,R_n$ are square-graph sums.  Combining
this identity with Proposition~\ref{prop:cubic} proves the square-graph cone
conjecture for $k=n-3$.
\end{proof}

The preceding uniform identity can be compressed when $n\ge5$ by using
the five-vertex sum $U_n$, associated with a doubled two-edge path
together with a disjoint doubled edge.

\begin{lemma}\label{lem:five-vertex-sum}
For centered variables one has
\begin{align*}
U_n={}&2(n-4)\Big(3(n-2)p_2^3+2(2n-5)p_3^2\\
&\hspace{2.4cm} +(n^2-9n+15)p_2p_4-n(n-1)p_6\Big).
\end{align*}
Equivalently,
\begin{equation}\label{eq:U-relation}
U_n=(n-4)\left(\frac12P_n+\frac14Q_n-(n-3)T_n\right).
\end{equation}
\end{lemma}

\begin{proof}
Apply the distinct-index M\"obius formula
\eqref{eq:mobius-distinct} to
$(z_i-z_j)^2(z_j-z_k)^2(z_l-z_m)^2$.  This gives the first displayed
formula.  Substitution of the formulas for $T_n,P_n,Q_n$ from
Lemma~\ref{lem:degree-six-sums} then gives \eqref{eq:U-relation}.
\end{proof}

\begin{theorem}[three-graph cubic certificate]\label{thm:cubic-threegraph}
For every $n\ge5$,
\begin{equation}\label{eq:cubic-threegraph}
 p_2^3-\frac{n(n-1)}{(n-2)^2}p_3^2
=
\frac{n+3}{n(n-2)}T_n
+\frac{1}{n(n-2)(n-4)}U_n
+\frac{1}{8n(n-2)^2}R_n .
\end{equation}
All three coefficients are positive.  In particular, the terminal cubic
certificate uses only three square-graph types for $n\ge5$.
\end{theorem}

\begin{proof}
In Proposition~\ref{prop:cubic-square-expansion}, the $P_n$ and $Q_n$ terms occur
in the combination
\[
\frac{1}{n(n-2)}\left(\frac12P_n+\frac14Q_n\right).
\]
Using \eqref{eq:U-relation}, this equals
\[
\frac{U_n}{n(n-2)(n-4)}+\frac{n-3}{n(n-2)}T_n.
\]
Adding the existing coefficient $6/(n(n-2))$ of $T_n$ gives
$(n+3)/(n(n-2))$, while the $R_n$ coefficient is unchanged.
\end{proof}

\begin{corollary}[the Alexandersson--Shapiro three-graph formula]
\label{cor:AS-threegraph}
Put $n=k+3$.  For $k\ge3$, let $g_1,g_2,g_3$ be respectively the doubled
triangle, the disjoint union of a doubled two-edge path and a doubled edge,
and the matching of three doubled edges, with isolated vertices adjoined up
to $k+3$ vertices.  Then
\begin{align*}
\D_k^{(k+3)}=(k!)^3\Bigg[&
\frac{(k+1)^3(k+2)(k+6)}{72}\,\widetilde g_1
+\frac{(k+1)^3k(k+2)}{12}\,\widetilde g_2\\
&+\frac{(k-1)k(k+1)^2(k+2)(k-2)}{96}\,\widetilde g_3
\Bigg].
\end{align*}
For $k=2$ the same formula holds with the last (zero) term omitted.  Thus the
three-graph formula conjectured in \cite[Example~3]{AS} holds whenever all
three graph types are present.
\end{corollary}

\begin{proof}
For $n\ge6$, the symmetrization convention of Section~\ref{sec:terminal} gives
\[
\widetilde g_1=6(n-3)!\,T_n,\qquad
\widetilde g_2=(n-5)!\,U_n,\qquad
\widetilde g_3=(n-6)!\,R_n.
\]
Combine Proposition~\ref{prop:cubic} and
Theorem~\ref{thm:cubic-threegraph}, substitute
$n=k+3$, and simplify the factorials.  For $k=2$ the matching coefficient
vanishes and the first two identities suffice.
\end{proof}

\begin{remark}[a graph-theoretic proof of the equality case]
The positive certificates also give a direct proof of
Corollary~\ref{cor:cubic-equality}.  Indeed, if the cubic moment inequality is an
equality, then every graph sum occurring with positive coefficient must
vanish.  Since $T_n$ is a sum of squared Vandermonde products over triples,
$T_n=0$ implies that the centered roots assume at most two distinct values.
For $n\ge5$, if both multiplicities were at least two, one of them would be at
least three; choosing an alternating doubled two-edge path using two vertices
from the larger class and one from the smaller, together with a disjoint edge
joining the two classes, gives a strictly positive summand of $U_n$.  Hence one
multiplicity is $1$.  For $n=4$ the same conclusion follows from the $P_n$ term
in Proposition~\ref{prop:cubic-square-expansion}, and for $n=3$ it is automatic.
Thus the square-graph certificate detects the sharp equality configurations
without an optimization argument.
\end{remark}

Theorem~\ref{thm:main-cubic} is the combination of
Proposition~\ref{prop:cubic-square-expansion},
Theorem~\ref{thm:cubic-threegraph}
and Corollary~\ref{cor:AS-threegraph}.
The normalization
$\Delta_{n,3}=\frac{108}{n^3(n-1)^3}
\bigl(p_2^3-\frac{n(n-1)}{(n-2)^2}p_3^2\bigr)$ converts
\eqref{eq:cubic-threegraph} into the introductory identity
\eqref{eq:intro-cubic-threegraph}.

\section{Where square graphs stop: the terminal quartic}\label{sec:threshold}

We now ask how long the square-graph phenomenon persists for the terminal
quartic.  Put
\[
   \Delta_{n,4}:=\disc(J_{n,4})
   =\left(\frac{24}{n!}\right)^6\D_{n-4}^{(n)},
\]
and let $\mathcal C_{n,12}\subset\mathrm{PST}_{n,12}$ denote the cone
generated by degree-$12$ square graphs.  Since the scalar relating
$\Delta_{n,4}$ and $\D_{n-4}^{(n)}$ is positive, the two polynomials have the
same cone-membership behavior.

In fixed degree, only finitely many reduced square-graph types are needed
to generate the cone, and membership is then a linear feasibility problem
in their coefficients.

\begin{lemma}[finite reduced graph types]\label{lem:finite-types}
Fix an even degree $d$.  There is a finite collection $\mathcal H_d$ of
square multigraphs without isolated vertices, each having at most $d$
vertices, such that for every $n\ge d$ the square-graph cone in
$\mathrm{PST}_{n,d}$ is generated by the symmetrizations of the graphs in
$\mathcal H_d$ after adjoining isolated vertices.  In the terminal problem of
order $r$ one may take $d=r(r-1)$.
\end{lemma}

\begin{proof}
If a square multigraph has total edge multiplicity $d$, every occupied
unordered pair of vertices has multiplicity at least two.  Hence there are at
most $d/2$ occupied pairs and at most $d$ nonisolated vertices.  After deleting
isolated vertices, only finitely many multigraphs can occur. For
$n\ge d$ every such reduced graph can be enlarged to $n$ vertices by adjoining
isolated vertices, which proves the claim.
\end{proof}

This finite list is compatible with the stable square-graph description
in \cite{AS}. For degree twelve and $n\ge12$, formula
\eqref{eq:orbit-monomial-coordinate} makes the generator coordinates
polynomial in $n$. Finding a square-graph certificate therefore means
solving a fixed finite system of coefficient equations with nonnegative
unknowns. For $\Delta_{n,4}$ this system is feasible exactly for
$n\le22$; for $n\ge23$ its infeasibility is witnessed by a linear
functional that is nonnegative on every graph generator and strictly
negative on $\Delta_{n,4}$.

For a homogeneous symmetric form $F$ of degree $12$ write
\[
       F=\sum_{\lambda\vdash12}c_\lambda(F)m_\lambda,
\]
where $m_\lambda$ is the monomial symmetric polynomial (terms with
$\ell(\lambda)>n$ are absent).  For $n\ge12$ define the linear functional
\begin{align}\label{eq:quartic-separator}
\Lambda_n(F)={}&96\bigl(
 c_{(12)}+c_{(10,2)}-c_{(9,3)}+c_{(8,4)}+c_{(8,2,2)}
 -c_{(7,5)}-c_{(7,3,2)} \notag\\
&\hspace{31mm}+c_{(6,6)}+c_{(5,4,3)}\bigr)
 +32c_{(4,4,4)}+(n-3)c_{(3,3,3,3)}.
\end{align}

\begin{theorem}[sharp terminal-quartic threshold]\label{thm:quartic-threshold}
For every integer $n\ge4$,
\[
             \Delta_{n,4}\in\mathcal C_{n,12}
             \quad\Longleftrightarrow\quad 4\le n\le22.
\]
Consequently Conjecture~\ref{conj:squaregraph} is false.  More precisely,
\[
       \D_{n-4}^{(n)}\notin\mathcal C_{n,12}
       \qquad(n\ge23),
\]
so $(n,k)=(23,19)$ is the first counterexample inside the terminal quartic
family.
\end{theorem}

The proof reduces to finitely many exact coefficient calculations. We
give the reduction and the separating functional explicitly.

Let $H$ be a loopless multigraph with six edges counted with multiplicity,
and let $2H$ denote the square graph obtained by doubling every edge.  If
$H$ has $v$ nonisolated labelled vertices, define its injective orbit sum by
\[
 \mathcal O_H^{(n)}=
 \sum_{\phi:[v]\hookrightarrow[n]}
 \prod_{\{u,v\}:\,m_{uv}>0}(x_{\phi(u)}-x_{\phi(v)})^{2m_{uv}}.
\]
After adjoining $n-v$ isolated vertices, the full $\Sn$-symmetrization is
$(n-v)!\mathcal O_H^{(n)}$; hence the two conventions generate the same ray.
Every degree-$12$ square graph arises in this way.

The orbit sums have a simple exact expansion in the monomial-symmetric
basis.  Put
\[
 Q_H(y_1,\ldots,y_v)=
 \prod_{\{u,v\}:\,m_{uv}>0}(y_u-y_v)^{2m_{uv}}
 =\sum_{\alpha\in\mathbb Z_{\ge0}^v}q_\alpha y^\alpha.
\]
For $\lambda\vdash12$, let $\alpha^+$ denote the positive entries of $\alpha$
and set
\[
 b_\lambda(H)=
 \left(\prod_{a\ge1}\operatorname{mult}_a(\lambda)!\right)
 \sum_{\operatorname{sort}(\alpha^+)=\lambda}q_\alpha .
\]
Then, whenever $n\ge v$,
\begin{equation}\label{eq:orbit-monomial-coordinate}
 [m_\lambda]\mathcal O_H^{(n)}
 =(n-\ell(\lambda))_{v-\ell(\lambda)}b_\lambda(H).
\end{equation}
Thus every coordinate is an explicit integer polynomial in $n$.

\begin{lemma}[dual certificate]\label{lem:quartic-dual}
For every degree-$12$ square graph and every $n\ge12$ one has
\[
                   \Lambda_n(\widetilde G)\ge0.
\]
On the other hand,
\begin{equation}\label{eq:quartic-target-separation}
 \Lambda_n(\Delta_{n,4})=
 -\frac{2654208\,Q(n)}
 {n^6(n-3)^2(n-2)^4(n-1)^5},
\end{equation}
where
\[
 Q(n)=n^6-36n^5+376n^4-1920n^3+5635n^2-9456n+6912.
\]
In particular, $\Lambda_n(\Delta_{n,4})<0$ for every $n\ge23$.
\end{lemma}

\begin{proof}
Dividing all even edge multiplicities by two identifies the degree-$12$
square-graph types with loopless multigraphs having six edges.  There are
$212$ such reduced isomorphism types; this is the enumeration already used in
\cite{AS}. For each type, expand \eqref{eq:orbit-monomial-coordinate}, apply
\eqref{eq:quartic-separator}, and substitute $n=m+12$. In all $212$ cases
the resulting polynomial in $m$ has nonnegative integer coefficients (in
$53$ cases it is identically zero).  Hence $\Lambda_n$ is nonnegative on every
generator for all $n\ge12$. The graph list and all $212$ coefficient
polynomials are included in the electronic supplement described in
Appendix~\ref{app:verification}.

Independently expanding the discriminant of the normalized quartic in
Proposition~\ref{prop:quartic} in the monomial-symmetric basis gives
\eqref{eq:quartic-target-separation}.  Finally,
\[
\begin{aligned}
Q(m+23)={}&m^6+102m^5+4171m^4+85572m^3\\
          &\quad+884074m^2+3748560m+957456,
\end{aligned}
\]
whose coefficients are all positive.  Therefore $Q(n)>0$ for $n\ge23$, and
the last assertion follows.
\end{proof}

\begin{proposition}[the exposed face at the threshold]
\label{prop:separator-face}
The cone $\mathcal C_{23,12}$ has dimension $21$, and its exposed face
\[
F_{23}:=\mathcal C_{23,12}\cap\ker\Lambda_{23}
\]
has dimension $12$ and codimension $9$. It is generated by the $53$
reduced graph types with identically zero separator polynomial.
\end{proposition}

\begin{proof}
Every nonzero separator polynomial in the proof of
Lemma~\ref{lem:quartic-dual} has nonnegative coefficients and is strictly
positive at $m=11$. Consequently, $\Lambda_{23}$ vanishes on exactly
the orbit sums of the $53$ indicated types. Since all generator evaluations are
nonnegative, a nonnegative combination belongs to $F_{23}$ precisely
when it uses only these zero rays.

In the $21$ centered power-sum monomials of weighted degree twelve,
form the coefficient matrix $A_{23}$ of all $212$ injective orbit sums
and its submatrix $A_{23}^0$ on the $53$ zero types. Exact rational row
reduction gives
\[
\operatorname{rank}A_{23}=21,\qquad
\operatorname{rank}A_{23}^0=12.
\]
The independent graph indices are recorded in the supplementary data
described in Appendix~\ref{app:verification}. These ranks give the asserted cone and
face dimensions. This statement concerns $n=23$; no uniform rank
assertion is needed for the threshold theorem.
\end{proof}

\begin{proof}[Proof of Theorem~\ref{thm:quartic-threshold}]
The non-membership for $n\ge23$ follows immediately from
Lemma~\ref{lem:quartic-dual} by separation.

For $4\le n\le22$, solving the coefficient equations gives identities
\[
\Delta_{n,4}=\sum_H c_{n,H}\mathcal O_H^{(n)},\qquad c_{n,H}\in\QQ_{>0},
\]
over a suitable set of graph types in the injective orbit normalization
above. These identities are
checked coefficientwise in centered power sums for $n\ge12$.
For $n<12$, Newton's identities reduce both sides to the algebraically
independent centered elementary symmetric functions $e_2,\ldots,e_n$.
Only graph types with at most $n$ nonisolated vertices are allowed.
The exact rational coefficients $c_{n,H}$ are provided in the
supplementary data described in Appendix~\ref{app:verification}.
No floating-point arithmetic enters this verification.
The numbers of graph rays used are
\[
\begin{array}{cccccccccc}
\toprule
 n&4&5&6&7&8&9&10&11&12\text{--}22\\
\midrule
 \#&1&3&6&8&10&11&12&12&13\\
\bottomrule
\end{array}
\]
Thus $\Delta_{n,4}\in\mathcal C_{n,12}$ throughout $4\le n\le22$, completing
the proof.
\end{proof}

This is the square-graph part of Theorem~\ref{thm:main-quartic}; the
sum-of-squares part is proved in Section~\ref{sec:quartic-sos}.

\begin{remark}
The obstruction begins at $n=23$, just before the convenient stable range
$n\ge24$ associated with degree $12$.  Thus the failure is not an artifact of
passing to an eventually stable ambient basis.  The functional
\eqref{eq:quartic-separator} already separates the whole infinite tail.
\end{remark}

The constant $2654208$ in \eqref{eq:quartic-target-separation} factors as
$27\cdot 8^4\cdot 4!$, and this is not a coincidence; the following
calculation shows what the separating functional measures.

\begin{remark}[leading asymptotics of the separator]
\label{rem:separator-asymptotics}
As $n\to\infty$,
\[
 \Lambda_n(\Delta_{n,4})=-\frac{2654208}{n^{11}}+O(n^{-12}),
\]
and the leading term is the contribution of the single monomial
$-27\beta^4$ of \eqref{eq:disc4}, entering $\Lambda_n$ through the
coefficient $(n-3)c_{(3,3,3,3)}$.

To see this, assign the weights $2,3,4$ to $\alpha,\beta,\gamma$.  Each of the six
monomials of \eqref{eq:disc4} has weight $12$.  Since
$\alpha=-6p_2n^{-2}(1+O(n^{-1}))$, $\beta=-8p_3n^{-3}(1+O(n^{-1}))$ and
$\gamma=3(p_2^2-2p_4)n^{-4}(1+O(n^{-1}))$, every one of the six monomials
is of the form $n^{-12}\,c\,\mu(p_2,p_3,p_4)\,(1+O(n^{-1}))$ with $c$ a fixed
integer and $\mu$ a weighted homogeneous polynomial in $p_2,p_3,p_4$ of
degree $12$.
Moreover, since $z_i=x_i-\barx$ and $\barx$ carries a factor $1/n$, the
monomial-symmetric coordinates in the variables $x$ of any fixed polynomial in
$p_2,p_3,p_4$ differ from those of the same polynomial in
$\sum x_i^2,\sum x_i^3,\sum x_i^4$ by $O(n^{-1})$.

Consequently every coordinate $c_\lambda(\Delta_{n,4})$ entering
\eqref{eq:quartic-separator} is $O(n^{-12})$, and all terms of $\Lambda_n$
with constant weight contribute $O(n^{-12})$ in total.  The only term with a
growing weight is $(n-3)c_{(3,3,3,3)}$.  Among the monomials of weighted
degree $12$ in $p_2,p_3,p_4$, the monomial $x_1^3x_2^3x_3^3x_4^3$ occurs only
in $p_3^4$, with coefficient $4!=24$; in every other monomial some factor
$p_2$ or $p_4$ would have to contribute an odd power.  Hence
$c_{(3,3,3,3)}(\Delta_{n,4})=-27\cdot 8^4\cdot 24\cdot n^{-12}+O(n^{-13})$,
and multiplying by $(n-3)$ gives the claim.

This calculation explains the eventual obstruction. Lemma~\ref{lem:quartic-dual}
shows that the constant-weight part of $\Lambda_n$ compensates for every
negative $(3,3,3,3)$-coordinate produced by a square graph. For the quartic
discriminant, the growing weight $(n-3)$ instead isolates the negative
$p_3^4$ contribution. The exact formula
\eqref{eq:quartic-target-separation}, rather than the asymptotic argument
alone, establishes the threshold $n=23$.
\end{remark}

\section{An SOS certificate beyond square graphs}\label{sec:quartic-sos}

We now prove the SOS part of Theorem~\ref{thm:main-quartic}.  Although a
general symmetry-adapted Gram matrix has thirteen blocks, the certificate
below uses only four of them and yields nine Reynolds-averaged squares.
Here a \emph{Reynolds-averaged square} is
$\Rey(f^2)=\frac1{n!}\sum_{\sigma\in\Sn}(\sigma f)^2$.
The number nine counts these averages; expanding each one gives an
ordinary sum of squares with a number of summands that can depend on $n$.

\begin{theorem}[nine-orbit SOS certificate]\label{thm:quartic-sos}
For every integer $n\geq23$, the form $\Delta_{n,4}$ is a sum of squares of
homogeneous degree-six polynomials in the centered roots.  Up to the action
of $\Sn$, the certificate uses nine Reynolds-averaged squares.  Consequently,
$\disc(P^{(n-4)})$ is a sum of squares in the differences $x_i-x_j$.
Together with the square-graph certificates for $4\leq n\leq22$, this proves
the same statement for every $n\geq4$.
\end{theorem}

We use the construction for symmetric sums of squares described by
Blekherman and Riener \cite[Section~4]{BlekhermanRiener}. It expresses an
invariant SOS through positive-semidefinite Gram blocks indexed by the
irreducible types of $\Sn$. Here the polynomials being squared have
degree six, and thirteen such types occur.
Matching coefficients gives linear equations in the entries of these
blocks; requiring the blocks to be positive semidefinite gives a
semidefinite feasibility problem.

Vanishing on coincidence loci supplies necessary linear restrictions on
the polynomials being squared. A numerical search, guided by these
restrictions and by the limiting Gram problem as $n\to\infty$, suggested
the four subspaces specified below, of dimensions $1,1,3,4$.

For these subspaces the coefficient identity has eighteen Gram entries
and thirteen independent equations. We specify five free coordinates
explicitly in \eqref{eq:free-gram-curve} and reconstruct the other
thirteen over $\QQ(n)$. The proof then verifies all twenty-one
coefficient identities and proves positivity of the four matrices from
their leading principal minors. An $LDL^{\mathsf T}$ decomposition
converts the blocks into $1+1+3+4=9$ square-orbits.

\subsection{Symmetry reduction and higher-Specht generators}
We recall the symmetry reduction in the form needed here.  The relevant
irreducible $\Sn$-types are indexed by the thirteen stable partitions
\[
\begin{split}
\Lambda_6=\{&(n),\,(n-1,1),\,(n-2,2),\,(n-2,1,1),\,(n-3,3),\\
&(n-3,2,1),\,(n-3,1,1,1),\,(n-4,4),\,(n-4,3,1),\\
&(n-4,2,2),\,(n-5,5),\,(n-5,4,1),\,(n-6,6)\}.
\end{split}
\]
The partition $\lambda$ records how a degree-six polynomial transforms when
the roots are permuted.  Higher Specht polynomials \cite{ATY}, multiplied by
invariant power-sum monomials of complementary degree, give a vector
$g_\lambda$ of multiplicity generators for each type.  Put
\[
 G_\lambda=\Rey(g_\lambda g_\lambda^{\mathsf T}),
 \qquad
 \Rey(f)=\frac1{n!}\sum_{\sigma\in\Sn}\sigma(f).
\]
An invariant Gram matrix does not mix inequivalent symmetry types.  The
symmetry-reduced Gram theorem \cite{GatermannParrilo,BlekhermanRiener}
therefore says that an invariant form $f$ is SOS if and only if
\[
 f=\sum_{\lambda\in\Lambda_6}\tr(\widetilde Q_\lambda G_\lambda),
 \qquad \widetilde Q_\lambda\succeq0.
\]

We next make the eleven generators used in our sparse certificate explicit.
For a partition $\mu\vdash12$, set
\[
 \mu[n]=(n-12+\mu_1,\mu_2,\ldots)
\]
and let $T_{\mu[n]}$ be the row-superstandard tableau of this shape, with row
sets $R(T)=(R_1,\ldots,R_\ell)$ and column group $C(T)$.  If
$\alpha=(\alpha_1;\ldots;\alpha_\ell)$ is a tuple of multisets of positive
exponents, let $m_R^\alpha$ be the product, over the rows, of the monomial
symmetric polynomial obtained by assigning the exponents in $\alpha_a$ to
distinct variables indexed by $R_a$; put $m_R^\varnothing=1$.  Define
\begin{equation}\label{eq:compressed-specht}
 F_{n;\mu,\alpha}(z)=
 \sum_{\pi\in C(T_{\mu[n]})}
 \operatorname{sgn}(\pi)\,m_{\pi R(T_{\mu[n]})}^{\alpha}(z).
\end{equation}
This is the row-symmetrized, column-antisymmetrized normalization of the
higher-Specht generators used in the certificate.  With
$p_\nu=\prod_{d\in\nu}p_d$ and $p_\varnothing=1$, the generators are:
\[
\begin{array}{cccl}
\toprule
\mu&j&\alpha&g_{\mu[n],j}\\
\midrule
(6,6)&1&(\varnothing;(1^6))&F_{n;\mu,\alpha}\\
(9,1,1,1)&1&(\varnothing;(1);(2);(3))&F_{n;\mu,\alpha}\\
(9,2,1)&1&(\varnothing;(1,1);(2))&p_2F_{n;\mu,\alpha}\\
&2&((1);(1,2);(2))&F_{n;\mu,\alpha}\\
&3&((2);(1,1);(2))&F_{n;\mu,\alpha}\\
&4&((1,1);(1,1);(2))&F_{n;\mu,\alpha}\\
(9,3)&1&(\varnothing;(1,1,1))&p_3F_{n;\mu,\alpha}\\
&2&((1);(1,1,1))&p_2F_{n;\mu,\alpha}\\
&3&((1);(1,2,2))&F_{n;\mu,\alpha}\\
&4&((1,1);(1,1,2))&F_{n;\mu,\alpha}\\
&5&((1,1,1);(1,1,1))&F_{n;\mu,\alpha}\\
\bottomrule
\end{array}
\]
Thus no numerical or unspecified basis choice enters the certificate.

\subsection{Vanishing conditions and the selected subspaces}
Lemma~\ref{lem:zero-locus} shows that every square in an SOS of
$\Delta_{n,4}$ vanishes whenever $n-2$ roots coincide. However, value
conditions alone do not explain the changes of basis below. By the
branching rule, an irreducible type $\lambda$ with $\lambda_1<n-2$ has no
$\mathfrak S_{n-2}$-fixed vector. Its evaluation at such a coincidence
configuration therefore vanishes automatically. All four types used in
our certificate have this property.

Higher-order vanishing supplies further necessary conditions.

\begin{lemma}[vanishing near an $(n-1)$-fold coincidence]
\label{lem:quartic-vanishing-order}
For $n\ge4$ and $y\in\bR^{n-1}$,
\begin{equation}\label{eq:quartic-vanishing-order}
\Delta_{n,4}(\epsilon y_1,\ldots,\epsilon y_{n-1},1)
=\left(\frac4n\right)^6\epsilon^6\Delta_{n-1,3}(y)
+O(\epsilon^7).
\end{equation}
Consequently, in every real SOS representation
$\Delta_{n,4}=\sum_a q_a^2$, each
$q_a(\epsilon y_1,\ldots,\epsilon y_{n-1},1)$ is $O(\epsilon^3)$.
\end{lemma}

\begin{proof}
Put $Q_\epsilon(t)=\prod_{i=1}^{n-1}(t-\epsilon y_i)$ and
$P_\epsilon(t)=(t-1)Q_\epsilon(t)$. The product rule, with the normalization
of Proposition~\ref{prop:terminal-normalized}, gives
\[
J_{n,4}(\epsilon s;\epsilon y,1)
=-\frac4n\epsilon^3J_{n-1,3}(s;y)+O(\epsilon^4).
\]
At $\epsilon=0$ the quartic is $t^3(t-4/n)$. For generic $y$, its three
small roots divided by $\epsilon$ tend to the roots of $J_{n-1,3}(s;y)$,
and its fourth root tends to $4/n$. The product of squared root differences
therefore has the leading term in \eqref{eq:quartic-vanishing-order}.
The equality of its coefficients extends to every $y$ by polynomial
identity. Finally, the lowest nonzero coefficient of a sum of real squares
cannot cancel. An $O(\epsilon^6)$ sum of squares forces each summand being
squared to have order at least three.
\end{proof}

The coefficients of orders zero, one and two in this restriction give
linear conditions on each degree-six polynomial. They are stronger than
its value at the coincidence point. We now specify the subspaces used in
the certificate explicitly, without claiming that these conditions exhaust
all possible Gram reductions.

Let
\[
 \Lambda_*=\{(6,6),(9,1,1,1),(9,2,1),(9,3)\},
 \qquad r_\mu=12-\mu_1.
\]
Their stabilized shapes are
\[
 (n-6,6),\qquad (n-3,1,1,1),\qquad (n-3,2,1),\qquad (n-3,3).
\]
Within each type write $g_j=g_{\mu[n],j}$.  For the scalar blocks put
\[
 h_{(6,6)}=(g_1),\qquad h_{(9,1,1,1)}=(g_1).
\]
For $\mu=(9,2,1)$ put
\begin{equation}\label{eq:facial-921}
\begin{aligned}
 h_1&=g_2+\frac{9-2n}{n(n-1)}g_1,&
 h_2&=g_3+\frac{6-n}{n(n-1)}g_1,&
 h_3&=g_4-\frac{(n-5)(n-6)}{2n(n-1)}g_1,
\end{aligned}
\end{equation}
and for $\mu=(9,3)$ put
\begin{equation}\label{eq:facial-93}
\begin{aligned}
 h_1&=g_2+\frac{n-6}{n-2}g_1,&
 h_2&=g_3-\frac{2(2n^2-15n+30)}{n(n-1)(n-2)}g_1,\\
 h_3&=g_4-\frac{(n-6)(n^2-11n+36)}{2n(n-1)(n-2)}g_1,&
 h_4&=g_5+\frac{(n-6)(n-7)(n-8)}{6n(n-1)(n-2)}g_1.
\end{aligned}
\end{equation}
The numbers of displayed generators and selected combinations are
\[
\begin{array}{ccccc}
\toprule
\mu&(6,6)&(9,1,1,1)&(9,2,1)&(9,3)\\
\midrule
M_\mu&1&1&4&5\\
m_\mu&1&1&3&4\\
\bottomrule
\end{array}
\]
Finally, set
\begin{equation}\label{eq:short-zonal}
 H_\mu(n;z)=\Rey(h_\mu h_\mu^{\mathsf T}),
 \qquad \mathcal H_\mu(n;z)=n^{r_\mu}H_\mu(n;z).
\end{equation}
Each $H_\mu(n;z)$ is a positive-semidefinite polynomial matrix, since for
every real vector $a$,
$a^{\mathsf T}H_\mu a=\Rey((a^{\mathsf T}h_\mu)^2)\geq0$.

\subsection{The four Gram blocks}

\begin{proposition}\label{prop:certificate}
For $\mu\in\Lambda_*$ there are symmetric matrices
\[
 \widehat Q_\mu(n)\in\operatorname{Sym}_{m_\mu}(\QQ(n)),
 \qquad (m_\mu)_{\mu\in\Lambda_*}=(1,1,3,4),
\]
such that, identically over $\QQ(n)$,
\begin{equation}\label{eq:certificate}
 n^{12}\Delta_{n,4}
 =\sum_{\mu\in\Lambda_*}
 \tr\!\left(\widehat Q_\mu(n)\mathcal H_\mu(n;z)\right).
\end{equation}
The rational matrices $\widehat Q_\mu(n)$ are positive definite for every
real $n\geq23$.  Consequently, \eqref{eq:certificate} gives an SOS for every
integer $n\geq23$.
\end{proposition}

We specify the eighteen rational entries by five free coordinates and
thirteen coefficient equations.  This gives a complete algebraic definition over
$\QQ(n)$ while keeping the presentation short.  Order the upper-triangular
entries as
\begin{equation}\label{eq:gram-coordinate-order}
\begin{split}
q=(&q^{66}_{11};q^{9111}_{11};
q^{921}_{11},q^{921}_{12},q^{921}_{13},q^{921}_{22},q^{921}_{23},q^{921}_{33};\\
&q^{93}_{11},q^{93}_{12},q^{93}_{13},q^{93}_{14},q^{93}_{22},
q^{93}_{23},q^{93}_{24},q^{93}_{33},q^{93}_{34},q^{93}_{44})^{\mathsf T}.
\end{split}
\end{equation}
Use as free coordinates
\[
 q_F=(q^{93}_{23},q^{93}_{24},q^{93}_{33},q^{93}_{34},q^{93}_{44})^{\mathsf T}.
\]
With $v=23/n$, choose
\begin{equation}\label{eq:free-gram-curve}
q_F(v)=a+bv+cv^2,
\end{equation}
where
\[
\begin{aligned}
a&=(303875,366590,457894,568944,720138)^{\mathsf T},\\
b&=(109069,131100,165275,209685,268735)^{\mathsf T},\\
c&=(10387,14221,17087,41066,106270)^{\mathsf T}.
\end{aligned}
\]
Let $q_P$ contain the other thirteen entries of $q$, and let
\begin{equation}\label{eq:pivot-monomials}
\begin{split}
\mathcal B_*=\{&p_2^6,p_2^4p_4,p_2^3p_3^2,p_2^3p_6,
p_2^2p_3p_5,p_2^2p_4^2,p_2p_3^2p_4,\\
&p_2p_4p_6,p_2p_5^2,p_3^4,p_3^2p_6,p_3p_4p_5,p_4^3\}.
\end{split}
\end{equation}
Assemble symmetric matrices $\widehat Q_\mu(q)$ from the coordinates $q$.
Coefficient extraction in the thirteen monomials of $\mathcal B_*$ from the
right-hand side of \eqref{eq:certificate} gives a linear system
\[
A_P(n)q_P+A_F(n)q_F=c_*(n),
\]
where $c_*(n)$ is the corresponding coefficient vector of
$n^{12}\Delta_{n,4}$.  The matrix $A_P(n)$ is nonsingular over $\QQ(n)$,
and the remaining entries are defined by
\begin{equation}\label{eq:pivot-reconstruction}
 q_P(n)=A_P(n)^{-1}\bigl(c_*(n)-A_F(n)q_F(23/n)\bigr).
\end{equation}
Equations \eqref{eq:compressed-specht}, \eqref{eq:short-zonal},
\eqref{eq:free-gram-curve}, and \eqref{eq:pivot-reconstruction} now define
the four rational matrices completely. Over $\QQ(n)$,
the thirteen pivot equations parameterize a five-dimensional affine slice
of the eighteen-coordinate Gram solution set; the remaining eight
weight-twelve coefficient equations vanish identically on this slice.
The first scalar block, for example, factors as
\begin{equation}\label{eq:first-short-block}
 \widehat Q_{(6,6)}(n)=
 \frac{432n^5(n-11)(n-9)(n-7)(n-5)(n-4)}
 {(n-10)(n-8)(n-6)(n-3)^2(n-2)(n-1)^4}.
\end{equation}
\begin{proof}[Proof of Proposition~\ref{prop:certificate}]
The polynomial identity is checked in the basis of the $21$ centered
power-sum monomials of weighted degree twelve.  The thirteen equations in
\eqref{eq:pivot-reconstruction} hold by construction, and exact substitution
in the remaining eight equations gives zero.  Thus all $21$ residuals vanish
identically in $\QQ(n)$.

Before specializing $n$, we check that the rational matrices have no
poles in the asserted range. Put
\begin{equation}\label{eq:gram-common-denominator}
d_{\mathrm{Gram}}(n)=
n^4(n-1)^6(n-2)^3(n-3)^4(n-5)^2(n-6)^3(n-8)^2(n-10)^2.
\end{equation}
Exact polynomial division shows that
$d_{\mathrm{Gram}}(n)\widehat Q_\mu(n)$ has entries in $\QQ[n]$ for
every block. Since $d_{\mathrm{Gram}}(n)>0$ for $n\ge23$, all eighteen
entries are defined throughout this range, and the rational-function
identity specializes there.

For positivity, let
\[
 \delta_{\mu,j}(n)=
 \det\widehat Q_\mu(n)[1\!:\!j,1\!:\!j],
 \qquad 1\leq j\leq m_\mu.
\]
After setting $s=n-23$, exact cancellation gives
\[
 \delta_{\mu,j}(23+s)=\frac{N_{\mu,j}(s)}{D_{\mu,j}(s)},
\]
and every coefficient of every $N_{\mu,j}$ and $D_{\mu,j}$ is strictly
positive.  The nine checks have degrees
\[
\begin{array}{ccc}
\toprule
\mu&j&(\deg N_{\mu,j},\deg D_{\mu,j})\\
\midrule
(6,6)&1&(10,10)\\
(9,1,1,1)&1&(20,20)\\
(9,2,1)&1,2,3&(22,22),(43,43),(64,64)\\
(9,3)&1,2,3,4&(24,24),(46,46),(60,60),(62,62)\\
\bottomrule
\end{array}
\]
Therefore all leading principal minors are positive for $s\geq0$.
Sylvester's criterion proves
$\widehat Q_\mu(n)\succ0$ for every real $n\geq23$.
The eighteen rational Gram entries and the nine reduced numerator and
denominator pairs are recorded in the supplementary data described in
Appendix~\ref{app:verification}. Their algebraic definitions are given
above; the expanded data allow the coefficient identities,
denominator clearing, and positivity checks to be carried out directly.
\end{proof}

The four positive-definite blocks can now be converted into an ordinary SOS.
Put
$\delta_{\mu,0}=1$ and, for $i>j$, define
\[
 L^\mu_{ij}(n)=
 \frac{\det\widehat Q_\mu[
 \{1,\ldots,j-1,i\},\{1,\ldots,j\}]}
 {\delta_{\mu,j}(n)},
 \qquad L^\mu_{ii}=1,
\]
with the remaining upper-triangular entries of $L^\mu$ equal to zero.  Then
\[
 \widehat Q_\mu=L^\mu
 \operatorname{diag}\!\left(
 \frac{\delta_{\mu,1}}{\delta_{\mu,0}},\ldots,
 \frac{\delta_{\mu,m_\mu}}{\delta_{\mu,m_\mu-1}}
 \right)(L^\mu)^{\mathsf T}.
\]
For $1\leq j\leq m_\mu$, set
\[
 f_{\mu,j}(n;z)=h_{\mu,j}+
 \sum_{i=j+1}^{m_\mu}L^\mu_{ij}(n)h_{\mu,i}.
\]
Dividing \eqref{eq:certificate} by $n^{12}$ now gives
\begin{equation}\label{eq:nine-square-certificate}
 \Delta_{n,4}=
 \sum_{\mu\in\Lambda_*}\sum_{j=1}^{m_\mu}
 n^{r_\mu-12}
 \frac{\delta_{\mu,j}(n)}{\delta_{\mu,j-1}(n)}
 \Rey\!\left(f_{\mu,j}(n;z)^2\right).
\end{equation}
All nine coefficients are positive for $n\geq23$.  Since the block sizes are
$1,1,3,4$, the formula has exactly nine terms up to symmetry.  Here
\[
 \Rey(f^2)=\frac1{n!}\sum_{\sigma\in\Sn}(\sigma f)^2,
\]
so expanding the Reynolds averages gives an ordinary finite SOS. The
number nine counts Reynolds-averaged squares in this certificate; no
minimality assertion is intended.

\begin{proof}[Proof of Theorem~\ref{thm:quartic-sos}]
Formula \eqref{eq:nine-square-certificate} proves the SOS assertion for
$\Delta_{n,4}$.  Moreover,
\[
 z_i=\frac1n\sum_{j=1}^n(x_i-x_j),
 \qquad
 \disc(P_x^{(n-4)})=\left(\frac{n!}{24}\right)^6\Delta_{n,4}.
\]
Thus all polynomials being squared are polynomials in the root differences,
and the normalization factor is itself a square.
\end{proof}

Together with Theorem~\ref{thm:quartic-threshold} this completes the proof
of Theorem~\ref{thm:main-quartic}.

\begin{corollary}[the two cones differ]\label{cor:cones-differ}
Let $n\ge23$.  The functional $\Lambda_n$ is nonnegative on
$\mathcal C_{n,12}$ but negative on the sum of squares
$\Delta_{n,4}\in\Sigma_{n,12}$.  In particular $\mathcal C_{n,12}$ is a
proper subcone of $\Sigma_{n,12}$, and at least one of the nine Reynolds
square-orbits in \eqref{eq:nine-square-certificate} does not belong to the
square-graph cone.
\end{corollary}

\begin{proof}
The first assertion is Lemma~\ref{lem:quartic-dual}.  Since
$\Lambda_n(\Delta_{n,4})<0$ and $\Delta_{n,4}$ is the positive sum of the
nine displayed summands, $\Lambda_n$ is negative on at least one of them.
\end{proof}

\begin{remark}[the abstract cones already differ in degree four]
\label{rem:degree-four-separation}
For four centered variables,
\[
F=p_4-\frac14p_2^2
=\sum_{i=1}^4\left(z_i^2-\frac{p_2}{4}\right)^2
\]
is a nonzero symmetric SOS. It vanishes at $z=(1,1,-1,-1)$.
Every degree-four square-graph generator is strictly positive at this
point: after halving edge multiplicities, the three reduced types are a
double edge, a two-edge path, and a two-edge matching, and each has an
injective placement across the two classes of equal coordinates.
Therefore $F\notin\mathcal C_{4,4}$. Degree two has only the ray $p_2$;
for $n\le3$ the space of symmetric translation-invariant quartics is
one-dimensional. Thus degree four is the smallest degree of abstract cone
separation, and four is the smallest number of variables in that degree.
The quartic-discriminant theorem concerns a distinguished degree-twelve
family and a sharp threshold within that family.
\end{remark}

The nine-orbit formula applies for $n\ge23$; the square-graph identities
of Theorem~\ref{thm:quartic-threshold} cover $4\le n\le22$.
The next section gives a Hermite-matrix factorization that proves the
sum-of-squares property uniformly for every $n\ge4$ by a single identity.

\section{Hermite-matrix certificates}\label{sec:hermite}

Hermite matrices provide a strengthening of the sum-of-squares
certificate. We establish explicit matrix factorizations for the
terminal cubic and quartic. The quartic identity gives a uniform scalar
SOS proof even where square-graph certificates fail.

For a monic polynomial $f$ of degree $r$ with roots
$\theta_1,\ldots,\theta_r$, put $s_j(f)=\sum_l\theta_l^j$. Its ordinary
Hermite matrix is
\[
H(f)=\bigl(s_{i+j}(f)\bigr)_{0\le i,j\le r-1}.
\]
It has polynomial entries in the coefficients of $f$, satisfies
$\det H(f)=\disc(f)$, and is positive semidefinite precisely when $f$ is
real-rooted. An \emph{SOS-matrix factorization} is an identity
$H(f)=A^{\mathsf T}A$ with polynomial entries in $A$. Cauchy--Binet then gives
\begin{equation}\label{eq:cauchy-binet}
\disc(f)=\sum_{|S|=r}\det(A_S)^2,
\end{equation}
where $A_S$ is the submatrix on the rows indexed by $S$. Thus a polynomial
factorization of the Hermite matrix gives a scalar SOS certificate.

For the first derivative, let $U$ have orthonormal columns spanning
$\mathbf1^\perp$ and set $M=U^{\mathsf T}\operatorname{diag}(x)U$.
The compression identity is a special case of Pereira's differentiator
theorem \cite[Theorem~2.5]{Pereira}: the unit vector
$v=\mathbf1/\sqrt n$ is a trace vector for every diagonal matrix, since
$v^{\mathsf T}\operatorname{diag}(x)^jv=\frac1n\sum_i x_i^j$.
Consequently,
\[
\det(tI-M)=P'(t)/n,\qquad
H(P'/n)=\bigl(\tr(M^{i+j})\bigr)_{0\le i,j\le n-2}.
\]
For completeness, the determinant identity also follows directly from
cofactors. Complete $v$ to the orthonormal basis $(v,U)$; the cofactor
corresponding to $v$ is the determinant of $tI-M$. Hence
\[
\det(tI-M)
=v^{\mathsf T}\operatorname{adj}(tI-\operatorname{diag}(x))v
=\frac1n\sum_{i=1}^n\prod_{j\ne i}(t-x_j)
=\frac1nP'(t).
\]
Since $M$ is symmetric,
$\tr(M^{i+j})=\langle M^i,M^j\rangle$ for the Frobenius inner product.
Hence $H(P'/n)$ is the Gram matrix of the polynomial vectors
$\operatorname{vec}(I),\operatorname{vec}(M),\ldots,\operatorname{vec}(M^{n-2})$.
This recovers the symmetric-matrix discriminant construction of
\cite{Ilyushechkin,Parlett,Domokos} and gives the first-derivative SOS.

For the terminal cubic, write $J_{n,3}(t)=t^3-st+b$, where
\[
s=\frac{3p_2}{n(n-1)},\qquad
b=-\frac{2p_3}{n(n-1)(n-2)},
\]
and set
\[
w_i=z_i^2-\frac{p_2}{n},\qquad
q_{ij}=z_iz_j+\frac{z_i^2+z_j^2}{n-2}
-\frac{p_2}{(n-1)(n-2)}\quad(i\ne j).
\]

\begin{theorem}[Hermite matrix of the terminal cubic]\label{thm:cubic-hermite}
For every $n\ge3$, the Hermite matrix of $J_{n,3}$ is an SOS matrix over
the polynomials in the centered roots. Explicitly,
\begin{equation}\label{eq:cubic-hermite}
\begin{aligned}
u^{\mathsf T}H(J_{n,3})u
={}&3\Bigl(u_0+\frac{2s}{3}u_2\Bigr)^2
+\frac{6}{n(n-1)}\sum_i
  \Bigl(z_iu_1+\frac{w_i}{n-2}u_2\Bigr)^2\\
&+\frac{6}{n^2(n-1)(n-2)}
  \sum_{i\ne j}q_{ij}^2u_2^2.
\end{aligned}
\end{equation}
\end{theorem}

\begin{proof}
Newton's identities give
\[
H(J_{n,3})=
\begin{pmatrix}
3&0&2s\\
0&2s&-3b\\
2s&-3b&2s^2
\end{pmatrix}.
\]
Using $\sum_i z_i=0$, direct expansion gives
\[
\sum_i z_iw_i=p_3,\qquad
\sum_i w_i^2=p_4-\frac{p_2^2}{n},
\]
and
\begin{equation}\label{eq:kurtosis-form}
\sum_{i\ne j}q_{ij}^2
=\frac{(n^2-3n+3)p_2^2-n(n-1)p_4}{(n-1)(n-2)}.
\end{equation}
Substitute these identities into the right-hand side of
\eqref{eq:cubic-hermite}. The $p_4$ terms in the coefficient of $u_2^2$
cancel, leaving $2s^2$; the other coefficients agree immediately with
the displayed matrix.
\end{proof}

\begin{corollary}[a scalar cubic SOS]\label{cor:cubic-hermite-scalar}
For every $n\ge3$, with $q_{ij}$ as above,
\[
\begin{aligned}
 p_2^3-\frac{n(n-1)}{(n-2)^2}p_3^2
 ={}&\frac{n(n-1)}{(n-2)^2}\sum_{i<j}(z_i-z_j)^2
       \Bigl(z_iz_j+\frac{p_2}{n}\Bigr)^2\\
 &+\frac{n-1}{n-2}\,p_2\sum_{i\ne j}q_{ij}^2.
\end{aligned}
\]
In particular, the terminal cubic discriminant is an SOS in the
differences of the original roots.
\end{corollary}

\begin{proof}
Apply \eqref{eq:cauchy-binet} to \eqref{eq:cubic-hermite}.
Every nonzero maximal minor uses the row
containing $u_0$. For the other two rows, use
\[
z_iw_j-z_jw_i=-(z_i-z_j)\left(z_iz_j+\frac{p_2}{n}\right),
\qquad \sum_i z_i^2=p_2,
\]
together with
\[
\disc(J_{n,3})=\frac{108}{n^3(n-1)^3}
\left(p_2^3-\frac{n(n-1)}{(n-2)^2}p_3^2\right).
\]
Since $p_2=\sum_i z_i^2$, the last term is also an SOS.
This proves the corollary and the cubic assertion of
Theorem~\ref{thm:main-hermite}.
\end{proof}

We turn to the terminal quartic. Write
$J_{n,4}(t)=t^4+\alpha t^2+\beta t+\gamma$ in centered coordinates,
with the coefficients of Proposition~\ref{prop:quartic}. Newton's
identities give
\[
\begin{gathered}
s_0=4,\quad s_1=0,\quad s_2=-2\alpha,\quad s_3=-3\beta,
\quad s_4=2\alpha^2-4\gamma,\\
s_5=5\alpha\beta,\qquad s_6=-2\alpha^3+6\alpha\gamma+3\beta^2.
\end{gathered}
\]
For $u=(u_0,u_1,u_2,u_3)^{\mathsf T}$ and
$\widetilde u=(u_1,u_2,u_3)^{\mathsf T}$, completing the square gives
\begin{equation}\label{eq:quartic-hermite-reduction}
u^{\mathsf T}H(J_{n,4})u
=4\left(u_0-\frac{\alpha}{2}u_2-\frac{3\beta}{4}u_3\right)^2
+\widetilde u^{\mathsf T}S_{n,4}\widetilde u,
\end{equation}
where
\begin{equation}\label{eq:quartic-reduced-hermite}
S_{n,4}=
\begin{pmatrix}
-2\alpha & -3\beta & 2\alpha^2-4\gamma\\
-3\beta & \alpha^2-4\gamma & \tfrac72\alpha\beta\\
2\alpha^2-4\gamma & \tfrac72\alpha\beta &
-2\alpha^3+6\alpha\gamma+\tfrac34\beta^2
\end{pmatrix}.
\end{equation}
The square completion is an invertible polynomial change of the $u$
coordinates. Hence $H(J_{n,4})$ is an SOS matrix if and only if
$S_{n,4}$ is an SOS matrix. The equivalent factorization
$S_{n,4}=B(z)^{\mathsf T}B(z)$ has column degrees $1,2,3$.
In particular,
\[
\Delta_{n,4}=4\det S_{n,4},\qquad
\D_{n-4}^{(n)}=4\left(\frac{n!}{24}\right)^6\det S_{n,4}.
\]
The second equality records the leading-coefficient normalization of
the unnormalized derivative. The positive scalar does not affect SOS
membership.

Retain $w_i$ and $q_{ij}$ from the cubic construction and define
\[
\begin{aligned}
\eta_i={}&\frac{2}{(n-2)(n-3)}
 \left(z_i^3-\frac{p_3}{n}\right)
 +\left(\frac{6}{n(n-1)}-\frac{1}{(n-2)(n-3)}\right)p_2z_i,\\
r_{ij}={}&z_iz_j(z_i+z_j)-\frac{p_2(z_i+z_j)}{n-2}
 +\frac{2(z_i^3+z_j^3)}{n-2}
 -\frac{2p_3}{(n-1)(n-2)}.
\end{aligned}
\]
The sum $R_n$ is the ordered six-index matching sum defined in the
introduction; it is zero at $n=5$.

\begin{theorem}[Hermite matrix of the terminal quartic]
\label{thm:quartic-hermite}
For every $n\ge4$, the ordinary Hermite matrix of $J_{n,4}$ is an SOS
matrix over the polynomials in the centered roots. For $n\ge5$, its
reduced block has the explicit identity
\begin{equation}\label{eq:quartic-hermite-factorization}
\begin{aligned}
\widetilde u^{\mathsf T}S_{n,4}\widetilde u
={}&\frac{12}{n(n-1)}\sum_i
 \left(z_i u_1+\frac{2w_i}{n-2}u_2+\eta_i u_3\right)^2\\
&+\frac{48}{n^2(n-1)(n-2)(n-3)(n-4)}
 \sum_{i<j}\left((n-4)q_{ij}u_2+r_{ij}u_3\right)^2\\
&+\frac{3R_n}{n^2(n-1)^2(n-2)^2(n-3)^2(n-4)}u_3^2.
\end{aligned}
\end{equation}
All three scalar weights are strictly positive for $n\ge5$.
Each term of $R_nu_3^2$ is the square of a product of three root
differences times $u_3$. Together with
\eqref{eq:quartic-hermite-reduction}, this is a polynomial matrix
factorization with column degrees $0,1,2,3$.
\end{theorem}

\begin{proof}
For $n=4$, $J_{4,4}$ is the original polynomial in centered roots,
so its Hermite matrix has the usual Vandermonde factorization.
For $n\ge5$, expansion using $\sum_i z_i=0$ gives
\[
\begin{aligned}
\sum_{i<j}q_{ij}^2
&=\frac{(n^2-3n+3)p_2^2-n(n-1)p_4}
 {2(n-1)(n-2)},\\
\sum_{i<j}q_{ij}r_{ij}
&=\frac{(n^2-2n+2)p_2p_3-n(n-1)p_5}
 {(n-1)(n-2)},\\
\sum_{i<j}r_{ij}^2
&=\frac{(n+2)p_2p_4-2np_6-p_2^3}{n-2}
 +\frac{(n^2-3n+4)p_3^2}{(n-1)(n-2)}.
\end{aligned}
\]
Combine these identities with the formula for $R_n$ in
Lemma~\ref{lem:degree-six-sums}. The first sum in
\eqref{eq:quartic-hermite-factorization} uses only
\[
\begin{gathered}
\sum_i z_iw_i=p_3,\quad
\sum_i z_i\left(z_i^3-\frac{p_3}{n}\right)=p_4,\quad
\sum_i w_i^2=p_4-\frac{p_2^2}{n},\\
\sum_i w_i\left(z_i^3-\frac{p_3}{n}\right)
 =p_5-\frac{p_2p_3}{n},\qquad
\sum_i\left(z_i^3-\frac{p_3}{n}\right)^2
 =p_6-\frac{p_3^2}{n}.
\end{gathered}
\]
Substitution verifies all six independent coefficients in
\eqref{eq:quartic-reduced-hermite}; the $p_5$ and $p_6$ terms cancel.
The matching-sum identity is valid also at $n=5$, when the sum is empty,
and every displayed denominator is nonzero for $n\ge5$. This proves
the factorization. All factors are polynomials in the centered roots
and hence in their differences. Cauchy--Binet now gives the scalar SOS
for every $n\ge4$, completing the quartic assertion of
Theorem~\ref{thm:main-hermite}.
\end{proof}

\begin{remark}[the first new finite instance]\label{rem:quartic-n6}
The pair $(n,k)=(6,2)$ is the first terminal-quartic case beyond
$k\in\{0,1,n-2,n-3\}$. Here the first two coefficients in
\eqref{eq:quartic-hermite-factorization} are $2/5$ and $1/90$, and
\[
\eta_i=\frac16\left(z_i^3-\frac{p_3}{6}\right)
 +\frac7{60}p_2z_i.
\]
After combining the $48$ identical ordered terms for each three-edge
matching, the last term is $\frac1{1800}u_3^2$ times the sum of the
$15$ squared matching products on six vertices. Thus the reduced
matrix has a factorization with $6+15+15=36$ rows, and the full Hermite
matrix has $37$ rows including the constant square. In particular,
there is no polynomial matrix-SOS obstruction at $n=6$, or at any
finite $n\ge4$ in this terminal family.
\end{remark}

The cubic and quartic factorizations suggest the following matrix
version of Conjecture~\ref{conj:sos}, which remains open for general
derivatives.

\begin{conjecture}[Hermite-matrix SOS conjecture]\label{conj:hermite}
For every $n\ge2$ and $0\le k\le n-2$, the ordinary Hermite matrix of
$P^{(k)}/\operatorname{lc}(P^{(k)})$ is an SOS matrix over
$\bR[x_1,\ldots,x_n]$.
\end{conjecture}

Since $P^{(k)}(t)=k!e_{n-k}(t\mathbf1-x)$, this is a question about the
ordinary Hermite matrix of an elementary symmetric polynomial in the fixed
direction $\mathbf1$.\footnote{We use the ordinary Hermite-matrix
condition of \cite{NPT}. In the notation of
\cite[Definition~4.2]{Saunderson}, the ordinary Newton-sum matrix is
$H_{p,e}(x)[e]$, whereas SOS-hyperbolicity requires $H_{p,e}(x)[u]$ to be
an SOS matrix for every $u\in\Lambda_+(p,e)$, with the base direction $e$
fixed.}
The conjecture implies the scalar SOS question by \eqref{eq:cauchy-binet}.
It holds for the original polynomial by the Vandermonde factorization,
for the first derivative by the compression above, and for the terminal
degrees two, three, and four: the quadratic Hermite matrix is
$\operatorname{diag}(2,2p_2/(n(n-1)))$, and the cubic and quartic are
Theorems~\ref{thm:cubic-hermite} and~\ref{thm:quartic-hermite}. The
remaining cases are $2\le k\le n-5$, beginning with $(n,k)=(7,2)$, where
the derivative has degree five.

\section{Further questions}\label{sec:limits}

The quartic theorem raises the question of what happens in higher terminal
degrees. We record two observations that guide this problem. The first
is that, in any fixed terminal degree, each membership question has the
same answer for all sufficiently large $n$.

\begin{proposition}\label{prop:eventual-constancy}
For each fixed $r\ge2$, both sets
\[
\{n\ge r:\Delta_{n,r}\in\mathcal C_{n,r(r-1)}\},\qquad
\{n\ge r:\Delta_{n,r}\in\Sigma_{n,r(r-1)}\}
\]
are either finite or cofinite in $\{n\ge r\}$.
\end{proposition}

\begin{proof}
In the stable range, either membership question is a feasibility problem
of fixed size with rational-function data in $n$: a linear program for
square graphs by Lemma~\ref{lem:finite-types}, and a symmetry-reduced
semidefinite program for SOS by \cite[Theorem~4.15]{BlekhermanRiener}.
Regard $n$ as a
real parameter, beyond the finitely many possible poles. Feasibility is a
semialgebraic condition by Tarski--Seidenberg. A semialgebraic subset of
the real line is a finite union of points and intervals, so its
intersection with the integers is eventually constant.
\end{proof}

Thus infinitely many SOS instances in a fixed terminal degree would imply
SOS for all sufficiently large $n$. The proposition alone does not decide
which eventual behavior occurs.

The second observation concerns the centered outlier limit.
Fix $a_1,\ldots,a_M$, append $n-M$ zeros, and center the resulting tuple
to obtain $z^{(n)}$. Put $A_1=\sum_i a_i$ and $D=d/ds$. The coefficient
formula \eqref{eq:terminal-coefficients} gives
\begin{equation}\label{eq:K}
n^rJ_{n,r}(s/n;z^{(n)})\longrightarrow
K_r(s):=\prod_{i=1}^M(1-a_iD)(s+A_1)^r.
\end{equation}
The shift by $A_1$ centers the limiting polynomial. For $j\ge2$, write
$p_j=\sum_i a_i^j$; these are the limits of the central power sums of
$z^{(n)}$. For example,
$K_4(s)=s^4-6p_2s^2-8p_3s+3(p_2^2-2p_4)$.

\begin{proposition}\label{prop:limit-hermite}
For every finite number of outliers and $3\le r\le6$, the ordinary Hermite
matrix of the centered polynomial $K_r$ is an SOS matrix over
$\bR[a_1,\ldots,a_M]$. Consequently $\disc(K_r)$ is a sum of squares.
\end{proposition}

The explicit factors and their exact verification are in
Appendix~\ref{app:outliers}. For the terminal quartic,
Theorem~\ref{thm:quartic-hermite} already supplies an explicit
finite-$n$ factorization for every $n\ge4$; the limiting certificate is
therefore consistent with a proved finite family. For $r=5,6$, the
outlier factors alone do not establish a finite-$n$ matrix certificate.
Any deformation argument must verify the finite-$n$ coefficient
equations and positivity on the relevant Gram face. A limiting SOS
alone does not establish eventual SOS.

\begin{samepage}
We close with three questions.
\begin{enumerate}[label={\rm(\arabic*)}]
\item Does the Hermite-matrix conjecture hold in terminal degrees
$r\ge5$? The first unresolved pair is $(n,k)=(7,2)$. Can the
outlier factors for $K_5$ and $K_6$ be extended to finite-$n$
factorizations, as the quartic identity does for $r=4$?

\item For each $r\ge5$, which $n$ give square-graph or SOS certificates
for $\Delta_{n,r}$? The first unresolved pair beyond the cases established
here is $(n,k)=(7,2)$. Appendix~\ref{app:quintic} gives the terminal
quintic and a factored slice as starting points.

\item Is there a conceptual explanation of the separator
\eqref{eq:quartic-separator} and the threshold $n=23$?
The $53$ graph types on which the separator vanishes may help identify
the relevant face of the square-graph cone: at $n=23$ they generate
the codimension-nine exposed face of
Proposition~\ref{prop:separator-face}. Does this face have dimension
$12$ for every $n\ge23$?
\end{enumerate}
\end{samepage}

\appendix
\begin{samepage}
\section{The terminal quintic and an odd slice}\label{app:quintic}
\begin{example}[the quintic terminal polynomial]\label{ex:quintic-terminal}
For $n\ge5$,
\[
\begin{aligned}
J_{n,5}(t)={}&t^5-\frac{10p_2}{n(n-1)}t^3
-\frac{20p_3}{n(n-1)(n-2)}t^2\\
&\quad+\frac{15p_2^2-30p_4}{n(n-1)(n-2)(n-3)}t
+\frac{20p_2p_3-24p_5}
{n(n-1)(n-2)(n-3)(n-4)} .
\end{aligned}
\]
Thus, for real roots, the discriminant of this quintic is a nonnegative
universal polynomial in $p_2,p_3,p_4,p_5$.
\end{example}
\end{samepage}

\begin{proof}
Apply Proposition~\ref{prop:terminal-normalized} with $r=5$ and use Newton's
identities in the centered case:
\[
e_2=-\frac{p_2}{2},\qquad e_3=\frac{p_3}{3},\qquad
e_4=\frac{p_2^2}{8}-\frac{p_4}{4},\qquad
e_5=\frac{p_5}{5}-\frac{p_2p_3}{6}.
\]
Substitution gives the displayed expression.
\end{proof}

\begin{corollary}[an odd quintic slice]\label{cor:quintic-odd-slice}
Assume $n\ge5$ and $p_3=p_5=0$.  Set
\[
 a=-\frac{10p_2}{n(n-1)},\qquad
 b=\frac{15(p_2^2-2p_4)}{n(n-1)(n-2)(n-3)}.
\]
Then
\[
 J_{n,5}(t)=t(t^4+at^2+b)
\]
and
\begin{equation}\label{eq:quintic-odd-discriminant}
       \disc(J_{n,5})=16b^3(a^2-4b)^2.
\end{equation}
Equivalently,
\[
\disc(J_{n,5})=
\frac{86400000\,(p_2^2-2p_4)^3
\bigl((n^2-11n+15)p_2^2+3n(n-1)p_4\bigr)^2}
{n^7(n-1)^7(n-2)^5(n-3)^5}.
\]
If the original roots are real, then
\begin{equation}\label{eq:odd-slice-moment-bounds}
 p_4\le\frac12p_2^2,
 \qquad
 (n^2-11n+15)p_2^2+3n(n-1)p_4\ge0.
\end{equation}
\end{corollary}

\begin{proof}
Under $p_3=p_5=0$, Example~\ref{ex:quintic-terminal} gives the displayed
odd quintic.  Since
$\disc\bigl(t(t^4+at^2+b)\bigr)=\Res(t,t^4+at^2+b)^2\disc(t^4+at^2+b)$
and $\Res(t,t^4+at^2+b)=b$, Proposition~\ref{prop:quartic} gives
\eqref{eq:quintic-odd-discriminant}.  Moreover
$a^2-4b=40\bigl((n^2-11n+15)p_2^2+3n(n-1)p_4\bigr)/(n^2(n-1)^2(n-2)(n-3))$.
For real original roots the polynomial $J_{n,5}$ is real-rooted by
Proposition~\ref{prop:terminal-normalized}.  Hence its nonzero roots occur as
$\pm u,\pm v$, so $b=u^2v^2\ge0$ and $a^2-4b=(u^2-v^2)^2\ge0$.  This proves
\eqref{eq:odd-slice-moment-bounds}.
\end{proof}

\section{Explicit factors in the centered outlier limit}\label{app:outliers}
Use the centered polynomial $K_r$ of \eqref{eq:K}, and count its roots
$\theta$ with multiplicity. Put
\[
g_u(t)=\sum_{p=1}^{r-1}u_pt^p,\qquad
\bar g_u=\frac1r\sum_\theta g_u(\theta),\qquad
G_r(u)=\sum_\theta\bigl(g_u(\theta)-\bar g_u\bigr)^2.
\]
For an ordered injective tuple $\tau=(i_1,\ldots,i_k)$, write
$a_\tau=\prod_{j=1}^k a_{i_j}$ and
$a[\tau]=(a_{i_1},\ldots,a_{i_k})$. Then, for $3\le r\le6$,
\begin{equation}\label{eq:limit-hermite}
G_r(u)=\sum_{k=1}^{r-1}(r)_{k+1}
\sum_{\tau\in\Inj(k)}a_\tau^2
\left(\sum_{q=k}^{r-1}\ell_{r,k,q}(a[\tau])u_q\right)^2.
\end{equation}
Here $\ell_{r,k,k}=1$; all other factors are listed in
Table~\ref{tab:ell}. The full Hermite form, for
$u=(u_0,\ldots,u_{r-1})$, is
\begin{equation}\label{eq:full-limit-hermite}
u^{\mathsf T}H(K_r)u=r\bigl(u_0+\bar g_u\bigr)^2+G_r(u).
\end{equation}
Cauchy--Binet applied to the reduced factor matrix expresses $\disc(K_r)$ as $r$ times
a sum of squared $(r-1)\times(r-1)$ minors.

In each tuple square, the coefficient
$a_\tau\ell_{r,k,q}(a[\tau])$ of $u_q$ is homogeneous of degree $q$
in the outliers, with $\deg p_j=j$. Thus the degrees agree with the
corresponding entries of the Hermite matrix. The explicit factors are
listed below and verified by coefficient comparison.

\begin{table}[htbp]
\[
\begin{array}{cl}
\toprule
 r=3 & \ell_{1,2}=y\\
\addlinespace
 r=4 & \ell_{1,2}=2y,\ \ \ell_{1,3}=2y^2+5p_2,\ \ \ell_{2,3}=e_1\\
\addlinespace
 r=5 & \ell_{1,2}=3y,\ \ \ell_{1,3}=6y^2+7p_2,\ \
       \ell_{1,4}=6y^3+27p_2y+18p_3,\\
     & \ell_{2,3}=2e_1,\ \ \ell_{2,4}=2h_2+9p_2,\ \ \ell_{3,4}=e_1\\
\addlinespace
 r=6 & \ell_{1,2}=4y,\ \ \ell_{1,3}=12y^2+9p_2,\ \
       \ell_{1,4}=24y^3+48p_2y+32p_3,\\
     & \ell_{1,5}=24y^4+168p_2y^2+152p_3y+93p_2^2+84p_4,\\
     & \ell_{2,3}=3e_1,\ \ \ell_{2,4}=6h_2+12p_2,\ \
       \ell_{2,5}=6h_3+42p_2e_1+38p_3,\\
     & \ell_{3,4}=2e_1,\ \ \ell_{3,5}=2h_2+14p_2,\ \ \ell_{4,5}=e_1\\
\bottomrule
\end{array}
\]
\caption{The polynomials $\ell_{r,k,q}$ of Proposition~\ref{prop:limit-hermite};
$y$ denotes the variable of a $1$-tuple, $e_1,h_j$ refer to the tuple
variables, and $p_j=\sum_i a_i^j$ ($j\ge2$) to all outliers.
The factors apply to the centered polynomial $K_r$ of \eqref{eq:K}.}
\label{tab:ell}
\end{table}

\begin{proof}[Exact verification of \eqref{eq:limit-hermite}]
Newton's identities express each $s_j(K_r)$ as a polynomial in
$p_2,\ldots,p_{2r-2}$ with rational coefficients. The coefficient matrix of $G_r$ has entries
$s_{i+j}-s_i s_j/r$ for $1\le i,j\le r-1$. On the right-hand side,
expand each tuple square and apply the distinct-index identity
\eqref{eq:mobius-distinct}. Comparing the resulting coefficients gives
zero residuals for every entry and each $r=3,4,5,6$. These calculations
use only rational arithmetic.
Equation~\eqref{eq:full-limit-hermite} follows by completing the constant
square, proving Proposition~\ref{prop:limit-hermite}.
\end{proof}

\begin{remark}\label{rem:limit-hermite}
The tuple blocks in this certificate have rank one. In every applicable
case of Table~\ref{tab:ell},
\[
\begin{aligned}
\ell_{r,k,k+1}&=(r-k-1)e_1(a[\tau]),\\
\ell_{r,k,k+2}&=(r-k-1)(r-k-2)h_2(a[\tau])
+\frac{(k+1)(2r-k-2)}2p_2.
\end{aligned}
\]
These patterns suggest extending the factorization to all $r$. For $r=3$
it is the limit of Theorem~\ref{thm:cubic-hermite}, with $z_i,w_i,q_{ij}$
replaced by $a_i,a_i^2,a_i a_j$.
\end{remark}

\begin{remark}[a random-matrix interpretation]
Let $v_i$ be independent uniform unit vectors in $\bR^r$, and put
$W=r\sum_i a_i v_i v_i^{\mathsf T}$. The rank-one determinant identity gives
\[
\mathbb E_v\det(sI-M-cvv^{\mathsf T})
=\left(1-\frac crD\right)\det(sI-M).
\]
Consequently the uncentered polynomial
$\prod_i(1-a_iD)s^r$ is $\mathbb E\det(sI-W)$, and the centered $K_r$
is $\mathbb E\det(sI-(W-A_1I))$. This fits the expected-characteristic-
polynomial framework of \cite{MSS}. It does not identify $H(K_r)$ with
the expectation of the Hermite matrices of the random characteristic
polynomials, since Newton sums depend nonlinearly on their coefficients.

The ordinary Newton-sum Hermite matrix is congruent to the Bezoutian of a
monic polynomial and its derivative; it is not generally that Bezoutian
in the monomial basis. Bilinearity of the Bezoutian and the subset-average
formula produce a double sum over pairs of subsets. The assertion that an
original root variable occurs with degree at most two applies to this
Bezoutian expansion, before the congruence, not to the Newton-sum matrix.
Organizing that expansion by common outliers is a possible route to a
general tuple-square formula.
\end{remark}

\section{Coefficient matrices and exact certificate data}\label{app:verification}

We specify the coefficient matrices for the quartic calculations.
All entries are determined by the formulas in the paper. The
supplementary data record the resulting exact certificates.

\subsection{The square-graph coefficient matrix}
Let
\[
\mathcal P_{12}=\{\lambda\vdash12:\lambda_i\ge2\text{ for all }i\},
\qquad p_\lambda=\prod_i p_{\lambda_i}.
\]
There are $21$ such partitions. For $n\ge12$ the $p_\lambda$ form a
basis of the symmetric translation-invariant forms of degree twelve.
Index columns by the $212$ reduced six-edge multigraphs $H$, and put
\begin{equation}\label{eq:graph-coefficient-matrix}
A(n)_{\lambda,H}=[p_\lambda]\mathcal O_H^{(n)},
\qquad b(n)_\lambda=[p_\lambda]\Delta_{n,4}.
\end{equation}
Thus $A(n)$ has size $21\times212$. Its entries are obtained by
expanding the squared edge products and applying
\eqref{eq:mobius-distinct}, with $p_0=n$ and $p_1=0$;
the target vector follows from Proposition~\ref{prop:quartic}.
For $4\le n<12$, use instead the basis
\[
\left\{e_2^{a_2}\cdots e_n^{a_n}:
  a_j\ge0,\ \sum_{j=2}^n ja_j=12\right\},
\]
where the $e_j$ are taken in the centered roots, and convert both
sides by Newton's identities. For every $n$, omit columns whose
graphs have more than $n$ nonisolated vertices. Denote the resulting
matrix and target by $A_n$ and $b_n$.

The positive certificates in Theorem~\ref{thm:quartic-threshold} are
rational vectors $c_n\ge0$ satisfying
\[
A_nc_n=b_n\qquad(4\le n\le22).
\]
The separator is checked by evaluating
$\rho_H(n)=\Lambda_n(\mathcal O_H^{(n)})$ using
\eqref{eq:orbit-monomial-coordinate}: every $\rho_H(12+m)$ has
nonnegative integer coefficients, and $53$ are identically zero.
At $n=23$, the matrix $A_{23}$ and its submatrix on these $53$ columns
have ranks $21$ and $12$, respectively.

\subsection{The Gram coefficient matrix}
Use the four polynomial matrices $\mathcal H_\mu$ of
\eqref{eq:short-zonal}, with block sizes $1,1,3,4$.
For $q\in\QQ(n)^{18}$, assemble the symmetric matrices
$\widehat Q_\mu(q)$ in the order \eqref{eq:gram-coordinate-order}, and set
\[
\mathcal L_n(q)=\sum_{\mu\in\Lambda_*}
 \tr\bigl(\widehat Q_\mu(q)\mathcal H_\mu(n;z)\bigr).
\]
If $\varepsilon_k$ is the $k$-th coordinate vector, the coefficient
matrix is
\begin{equation}\label{eq:gram-coefficient-matrix}
M(n)_{\lambda,k}=[p_\lambda]\mathcal L_n(\varepsilon_k),
\qquad
M(n)q(n)=n^{12}b(n).
\end{equation}
In particular, off-diagonal Gram entries contribute twice the
corresponding entry of $\mathcal H_\mu$. The matrix $M(n)$ has size
$21\times18$ and rank $13$ over $\QQ(n)$. The pivot rows
\eqref{eq:pivot-monomials}, the five prescribed coordinates
\eqref{eq:free-gram-curve}, and the reconstruction
\eqref{eq:pivot-reconstruction} determine $q(n)$ completely.
The remaining eight coefficient equations vanish identically.
The common denominator \eqref{eq:gram-common-denominator} excludes
poles for $n\ge23$; positivity follows from the nine leading principal
minors in the proof of Proposition~\ref{prop:certificate}.

\subsection{Supplementary data}
The electronic supplement \path{Discriminants_Certificates.zip}
contains a PDF presentation, algebraic text files, and CSV tables,
together with \path{README.txt}. The PDF \path{Certificate_Guide.pdf}
contains the complete exact data: the matrix assembly and eighteen
rational Gram entries, all finite square-graph coefficients for
$4\le n\le22$, the $212$ graph types and separator polynomials,
the nine leading principal minors with every numerator coefficient
and factored denominator, and the moment basis and rank column lists.
The same data are supplied in reusable form. The text files in
\path{algebra/} give the Gram matrices, all $212$
separator polynomials, the nine leading principal minors, and the
ordered moment basis and independent graph columns for the ranks
at $n=23$. Rational functions are written with factored denominators
and explicit numerator polynomials.

The tables in \path{csv/} record the graph edges, the finite positive
coefficients as reduced fractions, the separator coefficients indexed
by powers of $m=n-12$, and the positive principal-minor numerator
coefficients indexed by powers of $s=n-23$.
The Gram entries are functions of $n$ and follow the order
\eqref{eq:gram-coordinate-order}. The README specifies the notation
and table fields. All files give exact integers and rational numbers;
they can be checked by the coefficient comparisons and determinant
calculations described above.

Graph edges are recorded before doubling their multiplicities, and
all graph coefficients use the injective orbit sums
$\mathcal O_H^{(n)}$; full $\Sn$-symmetrization introduces the factor
$(n-v)!$ for a graph on $v$ vertices. Graph identifiers and vertex
labels are zero based. Each CSV coefficient is associated with its
explicit graph identifier or polynomial power.

The cubic and quartic Hermite factorizations are displayed in
Section~\ref{sec:hermite}, and all outlier-limit factors are in
Table~\ref{tab:ell}. Their verification requires only expansion,
Newton's identities, and the distinct-index formula
\eqref{eq:mobius-distinct}.

\begin{samepage}
\section*{Acknowledgments}
The authors used ChatGPT (OpenAI) to assist with drafting and revising the
exposition and with developing code for the construction and verification
of the computational certificates. AI-generated mathematical and
computational suggestions were treated as provisional; the certificate
identities and positivity claims were checked by exact rational and
symbolic computations, as described in Appendix~\ref{app:verification}.
The authors take full responsibility for the final text, mathematical
arguments, and computational certificates.
\par
\end{samepage}

\end{document}